\documentclass[11pt]{amsart}    % [a4paper, 11pt]
\usepackage{latexsym, amsmath, amssymb, amsfonts, amsthm, amssymb, setspace, verbatim}
\usepackage[all]{xy}
\usepackage{longtable}
\usepackage[ansinew]{inputenc}
\usepackage{hyperref}
\usepackage{enumerate}

\title[On the Baum-Connes spectral sequence for $\IZ^n$]{ On the spectral sequence associated with the Baum-Connes Conjecture for $\IZ^n$ }

\author[Sel\c{c}uk Barlak]{ Sel\c{c}uk Barlak }

\address{ Westf\"alische Wilhelms-Universit\"at, Fachbereich Mathematik, \phantom{--------------}\linebreak \text{}\hspace{3.2mm} Einsteinstrasse 62, 48149 M\"unster, Germany }
\email{ selcuk.barlak@uni-muenster.de }
\thanks{ \emph{Supported by:} SFB 878 \emph{Groups, Geometry and Actions}, GIF Grant 1137-30.6/2011, ERC AdG 267079}

\begin{document}

% math
\renewcommand\matrix[1]{\left(\begin{array}{*{10}{c}} #1 \end{array}\right)}  % Matrix
\newcommand\set[1]{\left\{#1\right\}}  % Menge
\newcommand\mset[1]{\left\{\!\!\left\{#1\right\}\!\!\right\}}

%% Besondere Variablen
%Zahlmengen-Stil
\newcommand{\IA}[0]{\mathbb{A}} \newcommand{\IB}[0]{\mathbb{B}}
\newcommand{\IC}[0]{\mathbb{C}} \newcommand{\ID}[0]{\mathbb{D}}
\newcommand{\IE}[0]{\mathbb{E}} \newcommand{\IF}[0]{\mathbb{F}}
\newcommand{\IG}[0]{\mathbb{G}} \newcommand{\IH}[0]{\mathbb{H}}
\newcommand{\II}[0]{\mathbb{I}} \renewcommand{\IJ}[0]{\mathbb{J}}
\newcommand{\IK}[0]{\mathbb{K}} \newcommand{\IL}[0]{\mathbb{L}}
\newcommand{\IM}[0]{\mathbb{M}} \newcommand{\IN}[0]{\mathbb{N}}
\newcommand{\IO}[0]{\mathbb{O}} \newcommand{\IP}[0]{\mathbb{P}}
\newcommand{\IQ}[0]{\mathbb{Q}} \newcommand{\IR}[0]{\mathbb{R}}
\newcommand{\IS}[0]{\mathbb{S}} \newcommand{\IT}[0]{\mathbb{T}}
\newcommand{\IU}[0]{\mathbb{U}} \newcommand{\IV}[0]{\mathbb{V}}
\newcommand{\IW}[0]{\mathbb{W}} \newcommand{\IX}[0]{\mathbb{X}}
\newcommand{\IY}[0]{\mathbb{Y}} \newcommand{\IZ}[0]{\mathbb{Z}}

%Geschwungener Stil
\newcommand{\CA}[0]{\mathcal{A}} \newcommand{\CB}[0]{\mathcal{B}}
\newcommand{\CC}[0]{\mathcal{C}} \newcommand{\CD}[0]{\mathcal{D}}
\newcommand{\CE}[0]{\mathcal{E}} \newcommand{\CF}[0]{\mathcal{F}}
\newcommand{\CG}[0]{\mathcal{G}} \newcommand{\CH}[0]{\mathcal{H}}
\newcommand{\CI}[0]{\mathcal{I}} \newcommand{\CJ}[0]{\mathcal{J}}
\newcommand{\CK}[0]{\mathcal{K}} \newcommand{\CL}[0]{\mathcal{L}}
\newcommand{\CM}[0]{\mathcal{M}} \newcommand{\CN}[0]{\mathcal{N}}
\newcommand{\CO}[0]{\mathcal{O}} \newcommand{\CP}[0]{\mathcal{P}}
\newcommand{\CQ}[0]{\mathcal{Q}} \newcommand{\CR}[0]{\mathcal{R}}
\newcommand{\CS}[0]{\mathcal{S}} \newcommand{\CT}[0]{\mathcal{T}}
\newcommand{\CU}[0]{\mathcal{U}} \newcommand{\CV}[0]{\mathcal{V}}
\newcommand{\CW}[0]{\mathcal{W}} \newcommand{\CX}[0]{\mathcal{X}}
\newcommand{\CY}[0]{\mathcal{Y}} \newcommand{\CZ}[0]{\mathcal{Z}}

%Script Stil
\newcommand{\FA}[0]{\mathfrak{A}} \newcommand{\FB}[0]{\mathfrak{B}}
\newcommand{\FC}[0]{\mathfrak{C}} \newcommand{\FD}[0]{\mathfrak{D}}
\newcommand{\FE}[0]{\mathfrak{E}} \newcommand{\FF}[0]{\mathfrak{F}}
\newcommand{\FG}[0]{\mathfrak{G}} \newcommand{\FH}[0]{\mathfrak{H}}
\newcommand{\FI}[0]{\mathfrak{I}} \newcommand{\FJ}[0]{\mathfrak{J}}
\newcommand{\FK}[0]{\mathfrak{K}} \newcommand{\FL}[0]{\mathfrak{L}}
\newcommand{\FM}[0]{\mathfrak{M}} \newcommand{\FN}[0]{\mathfrak{N}}
\newcommand{\FO}[0]{\mathfrak{O}} \newcommand{\FP}[0]{\mathfrak{P}}
\newcommand{\FQ}[0]{\mathfrak{Q}} \newcommand{\FR}[0]{\mathfrak{R}}
\newcommand{\FS}[0]{\mathfrak{S}} \newcommand{\FT}[0]{\mathfrak{T}}
\newcommand{\FU}[0]{\mathfrak{U}} \newcommand{\FV}[0]{\mathfrak{V}}
\newcommand{\FW}[0]{\mathfrak{W}} \newcommand{\FX}[0]{\mathfrak{X}}
\newcommand{\FY}[0]{\mathfrak{Y}} \newcommand{\FZ}[0]{\mathfrak{Z}}

\newcommand{\Fa}[0]{\mathfrak{a}} \newcommand{\Fb}[0]{\mathfrak{b}}
\newcommand{\Fc}[0]{\mathfrak{c}} \newcommand{\Fd}[0]{\mathfrak{d}}
\newcommand{\Fe}[0]{\mathfrak{e}} \newcommand{\Ff}[0]{\mathfrak{f}}
\newcommand{\Fg}[0]{\mathfrak{g}} \newcommand{\Fh}[0]{\mathfrak{h}}
\newcommand{\Fi}[0]{\mathfrak{i}} \newcommand{\Fj}[0]{\mathfrak{j}}
\newcommand{\Fk}[0]{\mathfrak{k}} \newcommand{\Fl}[0]{\mathfrak{l}}
\newcommand{\Fm}[0]{\mathfrak{m}} \newcommand{\Fn}[0]{\mathfrak{n}}
\newcommand{\Fo}[0]{\mathfrak{o}} \newcommand{\Fp}[0]{\mathfrak{p}}
\newcommand{\Fq}[0]{\mathfrak{q}} \newcommand{\Fr}[0]{\mathfrak{r}}
\newcommand{\Fs}[0]{\mathfrak{s}} \newcommand{\Ft}[0]{\mathfrak{t}}
\newcommand{\Fu}[0]{\mathfrak{u}} \newcommand{\Fv}[0]{\mathfrak{v}}
\newcommand{\Fw}[0]{\mathfrak{w}} \newcommand{\Fx}[0]{\mathfrak{x}}
\newcommand{\Fy}[0]{\mathfrak{y}} \newcommand{\Fz}[0]{\mathfrak{z}}

%Pfeilbefehle abkŸrzen
\newcommand{\Ra}[0]{\Rightarrow}
\newcommand{\La}[0]{\Leftarrow}
\newcommand{\LRa}[0]{\Leftrightarrow}

%Modifikation der Variablen
\renewcommand{\phi}[0]{\varphi}
\newcommand{\eps}[0]{\varepsilon}

%zusŠtzliche Features
\newcommand{\quer}[0]{\overline}
\newcommand{\uber}[0]{\choose}
\newcommand{\ord}[0]{\operatorname{ord}}		% Ordnung
\newcommand{\GL}[0]{\operatorname{GL}}
\newcommand{\supp}[0]{\operatorname{supp}}	% TrŠger
\newcommand{\id}[0]{\operatorname{id}}		% IdentitŠt
\newcommand{\Sp}[0]{\operatorname{Sp}}		% Spektrum eines Elements
\newcommand{\eins}[0]{\mathbf{1}}			% Eine Eins in allgemeinerem Kontext, z.B. in einem Ring
\newcommand{\diag}[0]{\operatorname{diag}}
\newcommand{\ind}[0]{\operatorname{ind}}
\newcommand{\auf}[1]{\quad\stackrel{#1}{\longrightarrow}\quad}
\newcommand{\hull}[0]{\operatorname{hull}}
\newcommand{\prim}[0]{\operatorname{Prim}}
\newcommand{\ad}[0]{\operatorname{Ad}}
\newcommand{\quot}[0]{\operatorname{Quot}}
\newcommand{\ext}[0]{\operatorname{Ext}}
\newcommand{\ev}[0]{\operatorname{ev}}
\newcommand{\fin}[0]{{\subset\!\!\!\subset}}
\newcommand{\diam}[0]{\operatorname{diam}}
\newcommand{\Hom}[0]{\operatorname{Hom}}
\newcommand{\Aut}[0]{\operatorname{Aut}}
\newcommand{\Inn}[0]{\operatorname{Inn}}
\newcommand{\del}[0]{\partial}
\newcommand{\dimeins}[0]{\dim^{\!+1}}
\newcommand{\dimnuc}[0]{\dim_{\mathrm{nuc}}}
\newcommand{\dimnuceins}[0]{\dimnuc^{\!+1}}
\newcommand{\dr}[0]{\operatorname{dr}}
\newcommand{\dimrok}[0]{\dim_{\mathrm{Rok}}}
\newcommand{\dimrokeins}[0]{\dimrok^{\!+1}}
\newcommand{\dreins}[0]{\dr^{\!+1}}
\newcommand*\onto{\ensuremath{\joinrel\relbar\joinrel\twoheadrightarrow}} % surjectiver Pfeil
\newcommand*\into{\ensuremath{\lhook\joinrel\relbar\joinrel\rightarrow}}  % injektiver Pfeil
\newcommand{\im}[0]{\operatorname{im}}
\newcommand{\dst}[0]{\displaystyle}
\newcommand{\cstar}[0]{$\mathrm{C}^*$}
\newcommand{\ann}[0]{\operatorname{Ann}}
\newcommand{\dist}[0]{\operatorname{dist}}
\newcommand{\asdim}[0]{\operatorname{asdim}}
\newcommand{\asdimeins}[0]{\operatorname{asdim}^{\!+1}}
\newcommand{\amdim}[0]{\dim_{\mathrm{am}}}
\newcommand{\amdimeins}[0]{\amdim^{\!+1}}
\newcommand{\dimrokc}[0]{\dim_{\mathrm{Rok,c}}}
\newcommand{\dimrokceins}[0]{\dimrokc^{\!+1}}
\newcommand{\act}[0]{\operatorname{Act}}
\newcommand{\idlat}[0]{\operatorname{IdLat}}
\newcommand{\Cu}[0]{\operatorname{Cu}}
\newcommand{\Ost}[0]{\CO_\infty^{\mathrm{st}}}
\newcommand{\op}[0]{\operatorname}
\newcommand{\kernel}[0]{\operatorname{ker}}
\newcommand{\coker}[0]{\operatorname{coker}}

% theorems
\newtheorem{satz}{Satz}[section]		% <--- optional, zŠhlt so mit den Abschnitten
\newtheorem{cor}[satz]{Corollary}
\newtheorem{lemma}[satz]{Lemma}
\newtheorem{prop}[satz]{Proposition}
\newtheorem{theorem}[satz]{Theorem}
\newtheorem*{theoreme}{Theorem}

\theoremstyle{definition}
\newtheorem{definition}[satz]{Definition}
\newtheorem*{definitione}{Definition}
\newtheorem{defprop}[satz]{Definition \& Proposition}
\newtheorem{nota}[satz]{Notation}
\newtheorem*{notae}{Notation}
\newtheorem{rem}[satz]{Remark}
\newtheorem*{reme}{Remark}
\newtheorem{example}[satz]{Example}
\newtheorem{defnot}[satz]{Definition \& Notation}
\newtheorem{question}[satz]{Question}
\newtheorem*{questione}{Question}
\newtheorem{construction}[satz]{Construction}

\newenvironment{bew}{\begin{proof}[Proof]}{\end{proof}}

\begin{abstract}
We examine a spectral sequence that is naturally associated with the Baum-Connes Conjecture with coefficients for $\IZ^n$ and also constitutes an instance of Kasparov's construction in his work on equivariant $KK$-theory. For $k\leq n$, we give a partial description of the $k$-th page differential of this spectral sequence, which takes into account the natural $\IZ^k$-subactions. In the special case that the action is trivial in $K$-theory, the associated second page differential is given by a formula involving the second page differentials of the canonical $\IZ^2$-subactions. For $n=2$, we give a concrete realisation of the second page differential in terms of Bott elements. We prove the existence of $\IZ^2$-actions, whose associated second page differentials are non-trivial. One class of examples is given by certain outer $\IZ^2$-actions on Kirchberg algebras, which act trivially on $KK$-theory. This relies on a classification result by Izumi and Matui. A second class of examples consists of certain pointwise inner $\IZ^2$-actions. One instance is given as a natural action on the group \cstar-algebra of the discrete Heisenberg group $H_3$. We also compute the $K$-theory of the corresponding crossed product. Moreover, a general and concrete construction yields various examples of pointwise inner $\IZ^2$-actions on amalgamated free product \cstar-algebras with non-trivial second page differentials. Among these, there are actions which are universal, in a suitable sense, for pointwise inner $\IZ^2$-actions with non-trivial second page differentials. We also compute the $K$-theory of the crossed products associated with these universal \cstar-dynamical systems.
\end{abstract}

\maketitle

%\thispagestyle{empty}
%\newpage \tableofcontents
%\setcounter{page}{1}

%%%%%%%%%%%%%%%%%%%%%%%%%%%%%%%%%%%%%%%%%%%%%%%%%%%%%%%%%%%%%%%%

\begin{samepage}
\tableofcontents
\end{samepage}

\setcounter{section}{-1}

\section{Introduction}
\noindent
The study of group actions on \cstar-algebras and their associated crossed product \cstar-algebras plays an important role within the field of operator algebra theory. Beside the fact that many interesting and prominent \cstar-algebras arise naturally as crossed products, their importance is also due to the various connections to other fields such as representation theory, index theory and topological dynamical systems.

One of the most important invariants for crossed product \cstar-algebras is topological $K$-theory. However, given an action of a locally compact group on a \cstar-algebra, it is often very difficult to compute the $K$-theory of the corresponding crossed product, even if the $K$-theory of the underlying \cstar-algebra is well understood. One approach to the computation of the $K$-theory of the reduced crossed product is proposed by the famous Baum-Connes Conjecture \cite{BaumConnes2000, BaumConnesHigson1994}. The conjecture in its general form involving coefficients predicts that for any second countable locally compact group $G$, any \cstar-algebra $A$, and any strongly continuous $G$-action $\alpha:G\curvearrowright A$, the assembly map
\[
\mu_A:K_*^{\op{top}}(G;A)\to K_*(A\rtimes_{\alpha,r} G)
\]
is an isomorphism, see also \cite{BaumConnesHigson1994}. This conjecture is known to hold for a strikingly large class of groups. In this context, let us emphasize the deep work of Higson and Kasparov \cite{HigsonKasparov2001} on groups with the Haagerup property, and of Lafforgue \cite{Lafforgue2012} on hyperbolic groups.

The left hand side, $K_*^{\op{top}}(G;A)$, is the topological $K$-theory for $G$ with coefficients in $A$. Beside others, the importance of the Baum-Connes Conjecture stems from the fact that the topological nature of $K_*^{\op{top}}(G;A)$ allows for computational tools, which are not evident to exist on the right hand side of the assembly map. Among others, there always exists a spectral sequence associated with $K_*^{\op{top}}(G;A)$.

In this paper, the spectral sequence for $G=\IZ^n$ is the object of our interest. Provided that $A$ is separable, an elegant \cstar-algebraic description of this spectral sequence can be derived from the formulation of the Baum-Connes Conjecture due to Meyer and Nest \cite{MeyerNest2006}. In their framework, the left hand side of the assembly map is given as $K_*(S^n(\CC_0(\IR^n)\otimes A)\rtimes_{\sigma\otimes\alpha}\IZ^n)$, where $\sigma$ is induced by the natural action of $\IZ^n$ on $\IR^n$. The natural filtration of $\IR^n$ by its skeletons yields a finite cofiltration of $(\CC_0(\IR^n)\otimes A)\rtimes_{\sigma\otimes\alpha}\IZ^n$ by \cstar-algebras. Given such a cofiltration, there is a standard procedure relying on Massey's technique of exact couples \cite{Massey1952, Massey1953}, which produces a spectral sequence converging to the $K$-theory of the cofiltrated \cstar-algebra. This is in great analogy to Schochet's \cite{Schochet1981} spectral sequence associated with filtrations of \cstar-algebras by closed ideals. In this way, the left hand side of the assembly map gives rise to a spectral sequence, which, in principle, allows us to compute the $K$-theory of the crossed product $A\rtimes_\alpha \IZ^n$.

The Baum-Connes spectral sequence for $\IZ^n$ is also induced by a natural cofiltration of the mapping torus $\CM_\alpha(A)$. In fact, the \cstar-algebras $\CM_\alpha(A)$ and $(\CC_0(\IR)\otimes A)\rtimes_{\sigma\otimes\alpha}\IZ^n$ are strongly Morita equivalent in a natural way by a result of Raeburn and Williams \cite{RaeburnWilliams1985}. As a consequence, the two spectral sequences are isomorphic. It seems that the description of the spectral sequence in terms of the mapping torus cofiltration has some advantages when it comes to investigating the occurring differentials. In the special case of a single automorphism, for example, this identification can be used to deduce the Pimsner-Voiculescu sequence \cite{PimsnerVoiculescu1980} from the left hand side of the Baum-Connes assembly map. The spectral sequence associated with the mapping torus cofiltration already appeared as a special case of Kasparov's much more general construction in \cite[6.10]{Kasparov1988}. So far, it has been used in different contexts, see for example \cite{BellissardKellendonkLegrand2001, SavignenBellissard2009}.

As always when working with spectral sequences, it is crucial, and in many cases very difficult, to understand the corresponding differentials. Therefore, this work is supposed to provide a starting point for a systematic investigation of the differentials of the Baum-Connes spectral sequence for $\IZ^n$. Our technical main result in this direction, Theorem \ref{partialDescriptionDifferentials}, yields a partial description of the $k$-th page differential in terms of the canonical $\IZ^k$-subactions. Although this description is far from being complete, it permits some interesting observations. For example, we obtain a complete description of the differential $d_1$ on the $E_1$-term. The pair $(E_1,d_1)$ turns out to coincide with the \emph{Pimsner-Voiculescu complex} defined by Savignen and Bellissard \cite{SavignenBellissard2009}, which reveals a striking similarity with the Pimsner-Voiculescu sequence. In particular, $d_1$ is completely determined by the induced $\IZ^n$-action in $K$-theory. This representation of the $E_1$-term also allows us to identify the $E_2$-term as the group cohomology of $\IZ^n$ with values in $K_*(A)$, which has already been pointed out by Kasparov in \cite[6.10]{Kasparov1988}.
Moreover, for a $\IZ^n$-action, which induces the trivial action on $K$-theory, we completely describe the second page differential in terms of the second page differentials of the canonical $\IZ^2$-subactions.

The main focus of this paper, however, is on the $K$-theory for crossed products by $\IZ^2$-actions. The case of $\IZ^2$-actions is still quite accessible via elementary methods, so that we are able to provide a concrete description of the corresponding second page differentials. This requires that we switch back from the topological perspective of the mapping torus to the algebraic perspective of the crossed product. Moreover, we discuss several instances of $\IZ^2$-actions on \cstar-algebras, whose associated second page differentials are non-trivial. This shows that, unlike for $\IZ$-actions, the $K$-theory for a crossed product by a $\IZ^2$-action is in general not determined, up to group extension problems, by the induced action on $K$-theory. Some of these examples are of independent interest, and we compute the $K$-theory of their corresponding crossed products. 

The existence of one large class of $\IZ^2$-actions with non-trivial second page differential turns out to be a consequence of Izumi's and Matui's classification of outer locally $KK$-trivial $\IZ^2$-actions on Kirchberg algebras, see \cite{IzumiMatui2010}. As a second class, we consider pointwise inner $\IZ^2$-actions. Contrary to the na\"{i}ve expectation, we find examples with non-trivial second page differential even within this class of actions. This is even more remarkable, as $\IZ^2$-actions arising from group representations into the unitary group of the underlying \cstar-algebra all give rise to isomorphic crossed products. An instructive example, which is also of interest in its own right, is given as a natural pointwise inner $\IZ^2$-action on the group \cstar-algebra of the discrete Heisenberg group $H_3$. This \cstar-algebra has already obtained a great deal of attention, whereat we point out the thorough investigation of Anderson and Paschke \cite{AndersonPaschke1989}. We conclude this paper by giving a general construction of pointwise inner $\IZ^2$-actions on certain amalgamated free product \cstar-algebras. Among these, we find actions which are universal, in a suitable sense, for non-trivial second page differentials coming from pointwise inner $\IZ^2$-actions. We also compute the $K$-theory of the crossed products associated with these universal actions.

The paper is organised as follows. In the first section we shortly recall the construction of the Baum-Connes spectral sequence and of Kasparov's spectral sequence for $\IZ^n$. Both spectral sequences are induced by natural cofiltrations of \cstar-algebras, and it turns out that they are isomorphic. The general machinery of producing a spectral sequence to a given finite cofiltration of \cstar-algebras and some further information on spectral sequences can be found in the appendix. 

In the second section, we investigate the differentials of the Baum-Connes spectral sequence. The definition of the spectral sequence using exact couples in the sense of Massey \cite{Massey1952,Massey1953} allows us to obtain a partial description of the $k$-th page differential in terms of the canonical $\IZ^k$-subactions. By applying our general result to the particular case $k=1$, we conclude that the $E_1$-term of the Baum-Connes spectral sequence coincides with the Pimsner-Voiculescu complex defined by Savignen and Bellissard \cite{SavignenBellissard2009}. We use this description to identify the $E_1$-term as a certain Koszul complex over the integral group ring of $\IZ^n$. With this description at hand, it is more or less standard to show that the $E_2$-term coincides with the group cohomology of $\IZ^n$ with values in $K_*(A)$, where  the $\IZ^n$-module structure is induced by $\alpha$. At the end of this section, we slightly extend our technical main result in the case that $k=2$. As a consequence, we get a complete description of the second page differential in terms of the second page differentials of the canonical $\IZ^2$-subactions, provided that the $\IZ^n$-action is trivial on $K$-theory.

In the third section, we provide concrete lifts for images under the boundary map of the Pimsner-Voiculescu sequence $\rho_*:K_*(A\rtimes_\alpha \IZ)\to K_{*+1}(A)$. The lifts for $\rho_1$ are well-known and easily found by using the partial isometry picture of the index map. They are all given as, what we call, generalised Bott elements associated with a commuting pair of a projection and a unitary. Finding suitable lifts for $\rho_0$ is more difficult. These are given as generalised Bott elements in the sense of Exel \cite{Exel1993}. To define these lifts, we use a result by Dadarlat \cite{Dadarlat1995}, which provides an alternative description of the $K_1$-group for a unital \cstar-algebra.

In the fourth section, we use the results of the third section and give a concrete description for the second page differential associated with a $\IZ^2$-action. Given a $\IZ^2$-action $\alpha$ on $A$ with canonical generators $\alpha_1$ and $\alpha_2$, the natural extension $\check{\alpha}_2\in \Aut(A\rtimes_{\alpha_1}\IZ)$ induces an endomorphism of the Pimsner-Voiculescu sequence for $\alpha_1$, which we suitably interpret as a short exact sequence. The second page differential associated with $\alpha$ then basically reduces to the Snake Lemma homomorphism of the corresponding diagram. Using this identification, we see that the image of this differential consists of generalised Bott-elements.

In the fifth section, we exploit Izumi's and Matui's classification result \cite{IzumiMatui2010} to show the existence of $\IZ^2$-actions on Kirchberg algebras, whose associated second page differentials do not vanish. Given a Kirchberg algebra $A$, we show that their classification invariant of a locally $KK$-trivial action $\alpha:\IZ^2\curvearrowright A$, which is an element in $KK(A,SA)$, descends to the associated second page differential, which basically amounts to a homomorphism $K_*(A)\to K_{*-1}(A)$. They prove that every element in $KK(A,SA)$ is realised as the invariant of such a $\IZ^2$-action, provided that $A$ is stable. If $A$ moreover satisfies the universal coefficient theorem (UCT) by Rosenberg and Schochet \cite{RosenbergSchochet1987}, then every homomorphism $K_*(A)\to K_{*-1}(A)$ occurs as the second page differential of some $\IZ^2$-action on $A$.

In the sixth section, we provide examples of pointwise inner $\IZ^2$-actions, which induce non-trivial second page differentials. After a general discussion on the second page differential of a pointwise inner $\IZ^2$-action, we consider the group \cstar-algebra of the discrete Heisenberg group \cstar-algebra, $\mathrm{C}^*(H_3)$, equipped with a natural pointwise inner action. This action is universal in the sense that every pointwise inner $\IZ^2$-action on a unital \cstar-algebra $B$ gives rise to an equivariant and unital $*$-homomorphism $B\to \mathrm{C}^*(H_3)$. We show that the associated second page differential is non-trivial and compute the $K$-theory of the corresponding crossed product. It turns out that this crossed product is not isomorphic in $K$-theory to the crossed product of $\mathrm{C}^*(H_3)$ by the trivial $\IZ^2$-action. Finally, we present a general method of constructing pointwise inner actions, which induce non-trivial second page differentials. All occurring \cstar-algebras are given as amalgamated free products of the form $A\ast_{\CC(\IT)}B$. We require that $A$ is equipped with a poinwise inner action, which has the property that the commutator of the two implementing unitaries has full spectrum. Moreover, $B$ is supposed to contain a central unitary with full spectrum. This central unitary gets identified under the amalgamation process with the abovementioned commutator and allows us to extend the action of $A$ to a pointwise inner action on $A\ast_{\CC(\IT)}B$. An additional, relatively mild $K$-theoretical assumption on $B$ ensures that the second page differential associated with the action on $A\ast_{\CC(\IT)}B$ is non-trivial. Among the constructed examples, we find \cstar-dynamical systems which are universal, in a suitable sense, for non-trivial second page differentials associated with pointwise inner $\IZ^2$-actions. We compute the $K$-theory of the crossed products associated with these universal \cstar-dynamical systems.

This paper is based on the main part of my doctoral thesis completed at Westf\"alische Wilhelms-Universit\"at M\"unster. I would like to thank my advisor Joachim Cuntz for his guidance and several helpful discussions. Moreover, I am grateful to Siegfried Echterhoff, Dominic Enders, Sven Raum, Nicolai Stammeier, G\'abor Szab\'o, Christian Voigt, Christoph Winges, and Wilhelm Winter for inspiring dicussions held in the course of my doctoral studies. I would also like to thank Dominic Enders, Sven Raum, Nicolai Stammeier, G\'abor Szab\'o, and Christoph Winges for their valuable feedback on preliminary versions. Finally, I would also like to thank Chris Phillips for pointing out to me the existence of certain $K$-theoretically trivial actions $\alpha:\IZ^2\curvearrowright A$ with the property that $A\rtimes_\alpha \IZ^2$ and $A\rtimes_{\id} \IZ^2$ have different $K$-theory.

%%%%%%%%%%%%%%%%%%%%%%%%%%%%%%%%%%%%%%%%%%%%%%%%%%%%%%%%%%%%%%%%

\section{Preliminaries}
\noindent
\begin{samepage}
\begin{nota} Unless specified otherwise, we will stick to the following notations throughout the paper.
\begin{itemize}
\item For a \cstar-algebra $A$, we write $\CM(A)$ for its multiplier algebra and $\CZ(A)$ for the its centre.
\item For a \cstar-algebra $A$ and $n\in \IN$, let $\CP_n(A)$ denote the set of projections in $M_n(A)$. If $A$ is unital, then $\CU_n(A)$ denotes the set of unitaries in $M_n(A)$.
\item For elements $a,b$ in a \cstar-algebra $A$, $[a,b]=ab-ba$ denotes commutator of $a$ and $b$.
\item If $H$ is a Hilbert space, then $\CK(H)$ denotes the \cstar-algebra of compact operators on $H$. Moreover, $\CK$ denotes the compact operators on a separable infinite dimensional Hilbert space.
\item $\CT$ denotes the Toeplitz algebra and $v\in \CT$ the canonical isometry.
\item We write $\set{e_1,\ldots,e_n}$ for the canonical $\IZ$-basis of $\IZ^n$.
\item For a \cstar-algebra $A$ and an action $\alpha:\IZ^n\curvearrowright A$, we write $\alpha_i=\alpha_{e_i}\in \Aut(A)$. Given an equivariant $*$-homomorphism $\phi:(A,\alpha,\IZ^n)\to (B,\beta,\IZ^n)$, we write $\check{\phi} : A\rtimes_\alpha \IZ^n\to B\rtimes_\beta \IZ^n$ for the natural extension of $\phi$.
\end{itemize}
\end{nota}
\end{samepage}

Let $A$ be a separable \cstar-algebra and $\alpha:\IZ^n\curvearrowright A$ an action by automorphisms. We shortly recall the definition of the associated Baum-Connes assembly map by using the framework of Meyer and Nest \cite{MeyerNest2006}. Let $\CC_0(\IR^n)$ be equipped with the $\IZ^n$-action $\sigma$ given by left translation, and consider $S^n$ and $\IC$ as trivial $\IZ^n$-algebras. Let $D\in KK^{\IZ^n}(S^n\CC_0(\IR^n),\IC)$ denote the Dirac operator and $D_A\in KK^{\IZ^n}(S^n\CC_0(\IR^n)\otimes A,A)$ the element obtained by taking the exterior product of $D$ with the identity on $A$. The assembly map is the homomorphism
\[
\mu_A:K_*(S^n(\CC_0(\IR^n)\otimes A)\rtimes_{\sigma\otimes\alpha} \IZ^n)\to K_*(A\rtimes_\alpha \IZ^n)
\]
induced by the image of $D_A$ under the descent homomorphism 
\[
KK^{\IZ^n}(S^n\CC_0(\IR^n)\otimes A,A)\to KK(S^n(\CC_0(\IR^n)\otimes A)\rtimes_{\sigma\otimes\alpha} \IZ^n,A\rtimes_\alpha \IZ^n).
\]
The validity of the Baum-Connes Conjecture with coefficients for $\IZ^n$ is a consequence of the fact that $D$ is invertible, which in turn follows from Kasparov's \cite{Kasparov1988} Dirac-dual Dirac method.

\begin{nota}
For a sequence of natural numbers $1\leq \mu_1<\ldots<\mu_k\leq n$, we write $\mu=(\mu_1,\ldots,\mu_k)$ for the induced $k$-tuple, and define $T(k,n)$ to be the set of all such ordered $k$-tuples. By convention, $T(0,n)$ is the set that contains only the empty tuple. We do not distinguish between a $1$-tuple $\lambda\in T(1,n)$ and the corresponding number $\lambda_1$.
\end{nota}

Consider the natural filtration of $\IR^n$ by its skeletons
\[
\emptyset=Y_{-1}\subset Y_0 \subset \cdots \subset Y_n = \IR^n, 
\]
that is,
\[
Y_k = \set{(x_1,\ldots,x_n)\in \IR^n\ :\ x_{\mu_1},\ldots,x_{\mu_{n-k}}\in \IZ \ \text{for some} \ \mu\in T(n-k,n)}.
\]
As each $Y_k\subset \IR^n$ is invariant under the canonical action of $\IZ^n$, the above filtration yields a finite cofiltration of \cstar-algebras

\begin{align}
\label{cofiltration K-homology}
\begin{xy}
\xymatrix{
(\CC_0(\IR^n)\otimes A)\rtimes_{\sigma\otimes\alpha} \IZ^n =: C_n \ar@{->>}[r] & \cdots \ar@{->>}[r] & (\CC_0(\IZ^n)\otimes A)\rtimes_{\sigma\otimes\alpha}\IZ^n=: C_0.
}
\end{xy}
\end{align}
A standard procedure, which we explain in Appendix \ref{Appendix}, yields a cohomological spectral sequence converging to $K_*((\CC_0(\IR^n)\otimes A)\rtimes_{\sigma\otimes\alpha} \IZ^n)\cong K_{*+n}(A\rtimes_\alpha\IZ^n)$.

\begin{definition}
Let $A$ be a separable \cstar-algebra and $\alpha:\IZ^n\curvearrowright A$ an action. We call the spectral sequence induced by the cofiltration \eqref{cofiltration K-homology} the \emph{Baum-Connes spectral sequence} for $\alpha$.
\end{definition}

A slightly different picture of this (or strictly speaking of an isomorphic) spectral sequence turns out to be better suited for our purposes. As $\IZ^n$ acts freely and properly on $\IR^n$, an imprimitivity result by Raeburn and Williams \cite[Theorem 2.2]{RaeburnWilliams1985} yields that $(\CC_0(\IR^n)\otimes A)\rtimes_{\sigma\otimes\alpha} \IZ^n$ is Morita equivalent to the mapping torus associated with $\alpha$, which by definition is the \cstar-algebra
\[
\CM_\alpha(A):=\left\lbrace f\in \CC(\IR^n, A)\ :\ f(x+z)=\alpha_{z}(f(x)),\ x\in \IR^n,\ z\in \IZ^n\right\rbrace.
\]
The conceptual reason for this lies in the fact that $\CM_\alpha(A)$ is the generalised fixed point algebra of the Rieffel proper $\IZ^n$-action $\sigma\otimes \alpha$, see
\cite{Rieffel1990}. Obviously, the restriction to $[0,1]^n\subset \IR^n$ induces a $*$-isomorphism between the mapping torus and the \cstar-algebra of functions $f\in \CC([0,1]^n,A)$ satisfying
\[
f(t_1,\ldots,t_{i-1},1,t_{i+1},\ldots,t_n) = \alpha_i(f(t_1,\ldots,t_{i-1},0,t_{i+1},\ldots,t_n))
\]
for all $t_1,\ldots,t_n\in [0,1]$. We will use this identification without further mentioning.

Consider the following filtration of the $n$-cube
\[
\emptyset=X_{-1}\subset X_0 \subset \cdots \subset X_n=[0,1]^n,
\]
where
\[
X_k:=\left\lbrace t\in [0,1]^n\ :\ t_{\mu_1}=\ldots=t_{\mu_{n-k}}=0\ \text{for some}\ \mu\in T(n-k,n)\right\rbrace.
\]
This filtration gives rise to a finite cofiltration of the mapping torus
\begin{align}
\label{cofiltrationMappingTorus}
\begin{xy}
	\xymatrix@C+0.2cm{
\CM_\alpha(A)=:F_n\ar@{->>}[r]^/0.3cm/{\pi_n} & F_{n-1} \ar@{->>}[r]^{\pi_{n-1}} & \cdots \ar@{->>}[r] & F_0:=A \ar@{->>}[r]^/-0.1cm/{\pi_0} & F_{-1}:=0,
}
\end{xy}
\end{align}
which in turn induces a spectral sequence converging to $K_*(\CM_\alpha(A))$. This spectral sequence is a special case of (the cohomological analogue of) a much more general construction due to Kasparov \cite[6.10]{Kasparov1988}.

\begin{prop}
Let $A$ be a separable \cstar-algebra and $\alpha:\IZ^n\curvearrowright A$ an action. Then the Baum-Connes spectral sequence for $\alpha$ is isomorphic to the spectral sequences associated with the mapping torus cofiltration \eqref{cofiltrationMappingTorus}.
\end{prop}
\begin{proof}
As $\IZ^n$ also acts freely and properly on $Y_k$ for $k=0,\ldots,n$, the result of Raeburn and Williams \cite[Theorem 2.2]{RaeburnWilliams1985} yields an imprimitivity bimodule $Z_k$ between $C_k$ and the \cstar-algebra
\[
\set{f\in \CC(Y_k,A)\ :\ f(x+z) = \alpha_z(f(x)),\ x\in Y_k,\ z\in \IZ^n},
\]
which is isomorphic to $F_k$ by restriction to $X_k\subset Y_k$. From the definition, one can deduce that these imprimitivity bimodules are compatible in the sense that for $k=1,\ldots,n$, there exists a surjective $\IC$-linear vector space homomorphism $Z_k\to Z_{k-1}$ respecting the module structures and the inner products. Let $L_k$ and $L_{k-1}$ denote the linking algebras for $Z_k$ and $Z_{k-1}$, respectively, see \cite{BrownGreenRieffel1977}. As the surjection $Z_k\to Z_{k-1}$ preserves all the structures of the imprimitivity bimodules, we obtain a surjective $*$-homomorphism $L_k\to L_{k-1}$ making the following diagram commute
\[
\xymatrix{
C_k \ar@{->>}[d] \ar@^{(->}[r] & L_k \ar@{->>}[d] & F_k \ar@{->>}[d] \ar@_{(->}[l]\\
C_{k-1} \ar@{^(->}[r] & L_{k-1} & F_{k-1} \ar@{_(->}[l]
}
\]
Brown's stabilisation theorem \cite{Brown1977} now yields isomorphisms $C_k\otimes \CK\stackrel{\cong}{\longrightarrow} F_k\otimes \CK$ making the following diagram of cofiltrations commute
\[
\xymatrix{
(\CC_0(\IR^n)\otimes A)\rtimes_{\sigma\otimes\alpha} \IZ^n \otimes \CK \ar@{->>}[r] \ar[d]^\cong & C_{n-1}\otimes \CK \ar[d]^\cong \ar@{->>}[r] & \cdots \ar@{->>}[r] & A\otimes \CK \ar[d]^\cong \\
\CM_\alpha(A)\otimes \CK\ar@{->>}[r] & F_{n-1}\otimes \CK \ar@{->>}[r] & \cdots \ar@{->>}[r] & A\otimes \CK
}
\]
Hence, the cofiltration preserving isomorphism $(\CC_0(\IR^n)\otimes A)\rtimes_{\sigma\otimes\alpha} \IZ^n\stackrel{\cong}{\longrightarrow} \CM_\alpha(A)\otimes \CK$ induces an isomorphism between the two spectral sequences associated with \eqref{cofiltration K-homology} and \eqref{cofiltrationMappingTorus}, see also Appendix \ref{Appendix}.
\end{proof}

In the following, we shall use this identification between these two spectral sequences without further mentioning.

Let us also mention that $K_*(\CM_\alpha(A))$ and $K_{*+n}(A\rtimes_\alpha \IZ^n)$ are isomorphic even if $A$ is not necessarily separable. This can be seen by combining a result by Olesen and Pedersen \cite[Theorem 2.4]{OlesenPedersen1986} with Takai duality \cite{Takai1975} and Connes' Thom isomorphism \cite{Connes1981}. A complete proof for this well-known result can be found in the author's doctoral thesis \cite[Section 1.2]{Barlak2014}. For this reason, we will in the following not restrict to separable \cstar-algebras. However, to keep notation simple, we will always use the term Baum-Connes spectral sequence, even if the underlying \cstar-algebra is not separable.  Given a $\IZ^n$-action, we shall, if not stated otherwise, denote by $(E_k,d_k)_{k\geq 1}$ the associated Baum-Connes spectral sequence.

%%%%%%%%%%%%%%%%%%%%%%%%%%%%%%%%%%%%%%%%%%%%%%%%%%%%%%%%%%%%%%%%%%%%%%%%%%%%%%%%%%%%%%%%%%%%%%%%%%%%%%%

\section{The differentials of the Baum-Connes spectral sequence}
\label{Section differentials}
\noindent
We shall start this section by examining the $E_1$-term of the Baum-Connes spectral sequence. For this, we have to introduce some further notation.

\begin{nota}
Let $n\in \IN$. If $k\leq l\leq n$ and if the underlying set of $\mu\in T(k,n)$ is contained in $\nu\in T(l,n)$, then we write $\mu\subseteq\nu$. In this situation, we define $\nu\setminus\mu\in T(l-k,n)$ as the unique element whose underlying set is the set difference of the underlying sets of $\nu$ and $\mu$. For $\mu\in T(k,n)$, let $\mu^\bot\in T(n-k,n)$ be the unique element disjoint to $\mu$. 

For $k=1,\ldots,n$, let $\Lambda^{k}(\IZ^n)$ be the $k$-th component of the exterior algebra over $\IZ^n$. We use the convention that $\Lambda^0(\IZ^n)=\IZ$ and $\Lambda^k(\IZ^{n})=0$ whenever $k<0$ or $k>n$. We write $e_\mu:=e_{\mu_1}\wedge\cdots\wedge e_{\mu_k}\in\Lambda^k(\IZ^n)$ for $\mu\in T(k,n)$. If $\mu\in T(0,n)$ is the empty tuple, then we define $e_\mu:=1$. We also agree on the convention that $e_\mu\wedge 1=1\wedge e_\mu =e_\mu$.
\end{nota}

\begin{rem}
The set $\left\lbrace e_\mu\ :\ \mu\in T(k,n)\right\rbrace$ defines a $\IZ$-basis for $\Lambda^k(\IZ^n)$. Hence, if we equip $T(k,n)$ with the lexicographical ordering, then the natural order-preserving bijection $T(k,n)\cong \left\lbrace 1,\ldots,\binom{n}{k}\right\rbrace$ yields a group isomorphism $\Lambda^k(\IZ^n)\cong \IZ^{\binom{n}{k}}$. Observe that this isomorphism exists for all $k\in \IZ$ since $\binom{n}{k}=0$ whenever $k<0$ or $k>n$. We will use these identifications throughout this section.
\end{rem}

Let $A$ be a \cstar-algebra, $\alpha:\IZ^n\curvearrowright A$ an action and let $(E_k,d_k)_{k\geq 1}$ denote the Baum-Connes spectral sequence for $\alpha$. Recall the mapping torus cofiltration \eqref{cofiltrationMappingTorus} associated with $\alpha$. We trivially extend this cofiltration by setting $F_p:=F_n$ for $p>n$, $F_p:=0$ for $p<-1$, and $\pi_p:=\id_{F_p}$ in either case. The ideal $I_p:=\op{ker}(\pi_p)$ gives rise to a short exact sequence
\[
\xymatrix{
0 \ar[r] & I_p \ar[r]^{\iota_p} & F_p \ar[r]^/-0.1cm/{\pi_p} & F_{p-1}\ar[r] & 0,
}
\]
whose associated boundary map is denoted by $\rho^{(p)}_*:K_*(F_{p-1})\to K_{*+1}(I_p)$.  We shall refer to $F_p$ as the \emph{$p$-skeleton} of $\CM_\alpha(A)$.

For $1\leq p\leq n$ and $\mu\in T(p,n)$, let $\alpha(\mu)$ denote the $\IZ^p$-action generated by $\alpha_{\mu_1},\ldots,\alpha_{\mu_p}$. Consider the associated mapping torus $\CM_{\alpha(\mu)}(A)$ and let $F(\mu)_k$ denote the $k$-skeleton of the respective cofiltration \eqref{cofiltrationMappingTorus}. Furthermore, we write $(E(\mu)_k,d(\mu)_k)_{k\geq1}$ for the spectral sequence associated with $(A,\alpha(\mu),\IZ^p)$. We obtain $\CM_{\alpha(\mu)}(A)$ as the quotient of $F_p$ under the surjective $*$-homomorphism induced by the restriction to the closed subset
\[
X(\mu):=\left\lbrace t\in [0,1]^n\ :\ t_{\mu^{\bot}_1}=\ldots=t_{\mu^{\bot}_{n-p}}=0\right\rbrace\subset X_p.
\]
This surjection fits into a commutative diagram
\begin{align}
\label{restrictionMappingTorus}
\begin{xy}
	\xymatrix{
	0 \ar[r] & I_p	\ar[r] \ar[d]^{\chi(\mu)} & F_p \ar[r] \ar@{->>}[d] & F_{p-1} \ar[r] \ar@{->>}[d]^{\pi(\mu)} & 0\\
	0 \ar[r] & S^p A \ar[r] & \CM_{\alpha(\mu)}(A) \ar[r] & F(\mu)_{p-1} \ar[r] & 0
	}
\end{xy}
\end{align}
and we write $\rho(\mu)_*:K_*(F(\mu)_{p-1})\to K_{*+1}(S^p A)$ for the boundary map associated with the lower row extension. Since 
\[
X_p=\bigcup\limits_{\mu\in T(p,n)}X(\mu),
\]
we conclude from diagram \eqref{restrictionMappingTorus} that $F_p$ is an iterative pullback of the $\binom{n}{p}$ many mapping tori of the natural $\IZ^p$-subactions glued together over the $\binom{n}{p-1}$ many mapping tori of the natural $\IZ^{p-1}$-subactions. This also shows that the $*$-homomorphism
\begin{align}
\label{identificationMultipleSuspensions}
\chi=(\chi(\mu))_{\mu\in T(p,n)}:I_p\stackrel{\cong}{\longrightarrow} (S^pA)^{\binom{n}{p}}
\end{align}
is an isomorphism. With the convention that $S^0A:=A$, the $E_1$-term is therefore given by
\[
E_1^{p,q}:=K_{p+q}(I_p)\cong  
\begin{cases}
K_{p+q}(S^pA)\otimes_\IZ \Lambda^p(\IZ^n) &,\ \text{for}\ 0\leq p\leq n,\\
0 &,\ \text{for}\ p<0\ \text{and}\ p>n.
\end{cases}
\]

We proceed with a description of the differentials $d_k$ in terms of the boundary maps $\rho(\mu)_*$. Observe that for a $\IZ^n$-action, the associated differential $d^{p,q}_k:E_k^{p,q}\to E_k^{p+k,q-k+1}$ can only be non-trivial if $0\leq p, p+k\leq n$, see also Appendix \ref{Appendix}. 

\begin{lemma}
\label{generalDescriptionDifferentials}
Assume that $0\leq p\leq n$ and $1\leq k\leq n-p$. Let $g\in E_{k}^{p,q}$ be represented by $x\in K_{p+q}(S^p A)\otimes_\IZ \Lambda^{p}(\IZ^{n})$. Let $y\in K_{p+q}(F_{p+k-1})$ be a lift for $K_{p+q}(\iota_p)(x)\in K_{p+q}(F_p)$ under the map induced by the surjection $F_{p+k-1}\to F_p$. For $\mu\in T(p+k,n)$, set 
\[
y_{\mu}:=K_{p+q}(\pi(\mu))(y)\in K_{p+q}(F(\mu)_{p+k-1}).
\]
Then the differential $d_k^{p,q}:E_k^{p,q}\to E_k^{p+k,q-k+1}$ satisfies
\[
d_k^{p,q}(g)=\left[\sum\limits_{\mu\in T(p+k,n)}\rho(\mu)_{p+q}(y_\mu)\otimes e_\mu\right]\in E_k^{p+k,q-k+1}.
\]
\end{lemma}
\begin{proof}
The definition of $(E_k,d_k)_{k\geq1}$ in terms of exact couples yields that $d_{k}^{p,q}(g)$ is given by the class of $\rho_{p+q}(y)$ in $E_{k}^{p+k,q-k+1}$, where 
\[
\rho_{p+q}:K_{p+q}(F_{p+k-1})\to K_{p+q+1}(I_{p+k})
\]
is the boundary map associated with the surjection $\pi_{p+k}:F_{p+k}\to F_{p+k-1}$, see also Remark \ref{inductive definition differentials}. For every $\mu\in T(p+k,n)$, the respective diagram \eqref{restrictionMappingTorus} together with the identification \eqref{identificationMultipleSuspensions} yields a commutative diagram
\[
\xymatrix@R+0.1cm{
K_*(F_{p+k-1}) \ar[r]^/-0.5cm/{\rho_*} \ar[d]_{K_{*}(\pi(\mu))} & K_{*+1}(S^{p+k}A)\otimes_\IZ \Lambda^p(\IZ^n) \ar[d]^{\op{pr}_\mu}\\
			K_{*}(F(\mu)_{p+k-1}) \ar[r]^{\rho(\mu)_{*}} & K_{*+1}(S^{p+k}A)
	}
\]
where $\op{pr}_\mu$ is the canonical projection onto the coordinate labelled by $\mu$. Thus,
\[\def\arraystretch{2.4}
\begin{array}{lcl}
d_{k}^{p,q}(g) & = & [\rho_{p+q}(y)]\\
 & = & \left[\sum\limits_{\mu\in T(p+k,n)}(\rho(\mu)_{p+q}\circ K_{p+q}(\pi(\mu)))(y)\otimes e_\mu\right]\\
 & = & \left[\sum\limits_{\mu\in T(p+k,n)}\rho(\mu)_{p+q}(y_\mu)\otimes e_\mu\right]\in E_k^{p+k,q-k+1}.
\end{array}
\]
\end{proof}

In \cite[Theorem 2]{SavignenBellissard2009}, Savignen and Bellissard define the \emph{Pimsner-Voiculescu complex} $(C_{PV},d_{PV})$ as
\[\begin{array}{lr}
C^{p,q}_{PV}:= K_{q}(A)\otimes_{\IZ}\Lambda^{p}(\IZ^{n}),& \\
d^{p,q}_{PV}:C_{PV}^{p,q}\to C_{PV}^{p+1,q}, & x\otimes e\mapsto \sum\limits_{k=1}^{n} (K_{q}(\alpha_k)-\id)(x)\otimes(e\wedge e_k)
\end{array}
\]
for $p,q\in \IZ$. Observe that Bott periodicity allows us to identify $E_1\cong C_{PV}$, which we shall do for the remainder of this section. They point out that this isomorphism actually intertwines the differentials $d_1$ and $d_{PV}$, so that the $E_2$-term is obtained as the cohomology of $(C_{PV},d_{PV})$. Moreover, they explicitly prove this for special situations they are interested in.

In the following we shall give a complete proof for the identification of $(E_1,d_1)$ with the Pimsner-Voiculescu complex. This turns out to be a consequence of this section's technical main result:

\begin{theorem}
\label{partialDescriptionDifferentials}
Let $0\leq p\leq n$, $1\leq k\leq n-p$, and $\mu\in T(p,n)$. Assume that $x\in K_q(A)$ represents an element $[x]\in E(\mu^\bot)^{0,q}_k$. Then $x\otimes e_\mu\in K_q(A)\otimes_\IZ\Lambda^p(\IZ^n)$ represents an element $[x\otimes e_\mu]\in E_k^{p,q}$. Its image under the differential $d_k^{p,q}:E_{k}^{p,q}\to E_{k}^{p+k,q-k+1}$ is given as follows. Let $y\in K_{q}(F(\mu^{\bot})_{k-1})$ be a lift for $x\in K_q(A)$ under the map induced by the surjection $F(\mu^{\bot})_{k-1}\to F(\mu^\bot)_0=A$. For $\lambda\in T(k,n)$, we set $y_\lambda\in K_q(F(\lambda)_{k-1})$ to be the image of $y$ under $K_q(F(\mu^\bot)_{k-1})\to K_q(F(\lambda)_{k-1})$ if $\lambda\subseteq \mu^{\bot}$, and zero otherwise. Then
\[
d_{k}^{p,q}([x\otimes e_\mu])=\left[\sum\limits_{\lambda\in T(k,n)}\rho(\lambda)_q(y_\lambda)\otimes(e_\mu\wedge e_\lambda)\right]\in E_k^{p+k,q-k+1}.
\]
\end{theorem}
\begin{proof}
Let $\sigma:\set{1,\ldots,n}\stackrel{\cong}{\longrightarrow}\set{1,\ldots,n}$ be the permutation given by
\[
\sigma^{-1}(l):=\begin{cases}
 \mu_{l} &\quad ,\ \text{if}\ 1\leq l\leq p,\\
\mu^\bot_{l-p} &\quad ,\ \text{if}\ p+1\leq l\leq n.
\end{cases}
\]
Consider the injective $*$-homomorphism
\[\def\arraystretch{1.4}
\begin{array}{l}
\iota:\CC_0((0,1)^{p}\times [0,1]^{n-p},A)\into \CC([0,1]^{n},A),\\
\iota(f)(t_1,\ldots,t_{n})=f(t_{\sigma(1)},\ldots,t_{\sigma(n)}),
\end{array}
\]
which gives rise to the following commutative diagram of cofiltrations
\begin{align}
\label{cofiltrationsSuspendedSystem}
\begin{xy}
	\xymatrix@C-0.15cm{
	S^p\CM_{\alpha(\mu^\bot)} \ar@{->>}[r] \ar@{^(->}[d]^\iota & S^p F(\mu^\bot)_{n-p-1} \ar@{->>}[r] \ar@{^(->}[d]^{\iota_{n-p-1}} & \cdots \ar@{->>}[r]  & S^p A \ar@{->>}[r] \ar@{^(->}[d]^{\iota_0} & 0 \ar[r] \ar[d] & \cdots \ar[r] & 0 \ar[d]\\
	\CM_\alpha(A)\ar@{->>}[r] & F_{n-1} \ar@{->>}[r] & \cdots \ar@{->>}[r] & F_{p} \ar@{->>}[r] & F_{p-1} \ar@{->>}[r] & \cdots \ar@{->>}[r] & A
	}
\end{xy}
\end{align}

Write $(\tilde{E}_k,\tilde{d}_k)_{k\geq1}$ for the spectral sequence associated with the upper row cofiltration after having applied Bott periodicity. Observe that $\iota$ gives rise to a morphism of spectral sequences $(E_k(\iota):\tilde{E_k}\to E_k)_{k\geq1}$. By construction, it holds that $E_k(\iota)([x])=\left[x\otimes e_\mu\right]\in E_k^{p,q}$.

Let $J_k:=\ker(S^pF(\mu^\bot)_k\to S^pF(\mu^\bot)_{k-1})$. For every $\nu\in T(p+k,n)$ with $\mu\subseteq \nu$, \eqref{cofiltrationsSuspendedSystem} gives rise to a commutative diagram with exact rows

\[
\xymatrix@C-0.8cm{
	J_k \ar@{^(->}[rr] \ar@{^(->}[dd]_/-0.3cm/{\eta} \ar@{->>}[rd] && S^pF(\mu^\bot)_k \ar@{->>}[rr] \ar@{^(->}[dd]|(0.5)\hole_/-0.6cm/{\iota_k} \ar@{->>}[rd] && S^pF(\mu^\bot)_{k-1} \ar@{^(->}[dd]|(0.5)\hole_/-0.6cm/{\iota_{k-1}} \ar@{->>}[rd] & \\	 	
	 	& S^p(S^k A) \ar@{^(->}[rr] \ar@{^(->}[dd]^/-0.5cm/{\eta(\nu)} && S^p\CM_{\alpha(\nu\setminus\mu)}(A) \ar@{->>}[rr] \ar@{^(->}[dd]^/-0.5cm/{\kappa} & & S^pF(\nu\setminus\mu)_{k-1} \ar@{^(->}[dd] \\
		I_{p+k} \ar@{^(->}[rr]|(0.42)\hole \ar@{->>}[rd] && F_{p+k} \ar@{->>}[rd] \ar@{->>}[rr]|(0.48)\hole && F_{p+k-1} \ar@{->>}[rd] &&\\
		 &	S^{p+k}A \ar@{^(->}[rr] && \CM_{\alpha(\nu)}(A) \ar@{->>}[rr] && F(\nu)_{p+k-1} &
	}
\]
As observed before, the maps
\[
J_k\cong (S^p(S^k A))^{\binom{n-p}{k}}\to S^p(S^k A)\quad \text{and}\quad
I_{p+k}\cong (S^{p+k}A)^{\binom{n}{p+k}}\to S^{p+k}A
\]
are the canonical surjections onto the coordinate labelled by $\nu\setminus\mu$ and $\nu$, respectively.
Let $\sigma(\nu):\set{1,\ldots,p+k}\stackrel{\cong}{\longrightarrow}\set{1,\ldots,p+k}$ be the permutation given by
\[
\sigma(\nu)^{-1}(l):=\begin{cases}
\mu_l &\quad ,\ \text{if}\ 1\leq l\leq p,\\
(\nu\setminus \mu)_{l-p} &\quad ,\ \text{if}\ p+1\leq l\leq p+k. \\

\end{cases}
\]
As above, $\kappa: S^p\CM_{\alpha(\nu\setminus\mu)}(A)\to\CM_{\alpha(\nu)}(A)$ is induced by the injective $*$-ho\-mo\-mor\-phism
\[\def\arraystretch{1.4}
\begin{array}{l}
\kappa:\CC_0((0,1)^p\times [0,1]^{p+k},A)\to \CC([0,1]^n,A),\\
\kappa(f)(t_1,\ldots,t_n)=f(t_{\sigma(\nu)(1)},\ldots,t_{\sigma(\nu)(n)}).
\end{array}
\]
Hence, by using  the canonical isomorphism $S^p(S^kA)\cong S^{p+k} A$, we see that the $*$-automorphism $\eta(\nu):S^{p+k}A\stackrel{\cong}{\longrightarrow} S^{p+k}A$ is induced by the homeomorphism of $\IR^{p+k}$ which permutes the coordinates via $\sigma(\nu)$. It therefore follows that $K_*(\eta(\nu))=\op{sgn}(\sigma(\nu))\cdot \id$. Observe that this also shows that
\[
K_*(\eta):K_*(S^{p+k}A)\otimes_\IZ \Lambda^k(\IZ^{n-p})\to K_*(S^{p+k}A)\otimes_\IZ \Lambda^{p+k}(\IZ^n)
\]
is injective. Using Lemma \ref{generalDescriptionDifferentials} and the fact that 
\[
e_{\nu}=\op{sgn}(\sigma(\nu)) \cdot e_{\mu}\wedge e_{\nu\setminus \mu}\in \Lambda^{p+k}(\IZ^{n}), 
\]
we conclude that 
\[\def\arraystretch{3}
\begin{array}{lcl}
d_k^{p,q}([x\otimes e_\mu])  & = & d_k^{p,q}(E_k^{p,q}(\iota)([x]))\\
 & = & E_k^{p,q}(\iota)(\tilde{d}_{k}^{p,q}([x]))\\
 & = & E_k^{p,q}(\iota)\left(\left[\sum\limits_{\begin{smallmatrix}\lambda\in T(k,n):\\\lambda\subseteq \mu^\bot\end{smallmatrix}}\rho(\lambda)_q(y_\lambda)\otimes e_{\lambda}\right]\right)\\
  & = & \left[K_{q+1}(\eta)\left(\sum\limits_{\begin{smallmatrix}\lambda\in T(k,n):\\\lambda\subseteq \mu^\bot\end{smallmatrix}}\rho(\lambda)_q(y_\lambda)\otimes e_{\lambda}\right)\right]\\
 & = & \left[\sum\limits_{\begin{smallmatrix}\nu\in T(p+k,n):\\\mu\subseteq \nu\end{smallmatrix}}K_{q+1}(\eta(\nu))(\rho(\nu\setminus \mu)_q(y_{\nu\setminus \mu}))\otimes e_\nu\right]\\
 & = & \left[\sum\limits_{\lambda\in T(k,n)}\rho(\lambda)_q(y_\lambda)\otimes (e_\mu\wedge e_\lambda)\right]\in E_k^{p+k,q-k+1}.
\end{array}
\]
\end{proof}

\begin{cor}
The isomorphism $E_1\cong C_{PV}$ intertwines the differentials $d_1$ and $d_{PV}$.
\end{cor}
\begin{proof}
If $k=1$, then the assumptions in Theorem \ref{partialDescriptionDifferentials} are satisfied for every $\mu\in T(1,n)$ and $x\in K_q(A)$. For each $\lambda\in T(1,n)$, $\rho(\lambda)_*$ is the boundary map associated with the six-term exact sequence of the mapping torus extension
\[
\xymatrix{
0 \ar[r] & SA \ar[r] & \CM_{\alpha_\lambda}(A) \ar[r]^/0.25cm/{\op{ev}_0} & A \ar[r] & 0.
	}
\]
Hence, by Bott periodicity, we obtain $\rho(\lambda)_* = K_*(\alpha_\lambda)-\id$, and Theorem \ref{partialDescriptionDifferentials} then yields
\[
d_{1}^{p,q}(x\otimes e_\mu)=\sum\limits_{\lambda=1}^{n}(K_q(\alpha_\lambda)-\id)(x)\otimes(e_\mu\wedge e_\lambda).
\]
\end{proof}

In \cite[6.10]{Kasparov1988}, Kasparov points out that the $E_2$-term of the spectral sequence associated with \eqref{cofiltrationMappingTorus} is given by the group cohomology of $\IZ^n$ with values in $K_*(A)$ (where the $\IZ^n$-module structure is induced by $\alpha$). For the reader's convenience, we attach a complete proof for this here. For the definition and important properties of group cohomology, the reader is referred to \cite{Brown1982}.

Let $R:=\IZ[\IZ^n]$ be the integral group ring of $\IZ^n$ and $t_i\in \IZ^n\subseteq R$ denote the $i$-th canonical basis element for $i=1,\ldots,n$. The \emph{Koszul complex} $(G_*,g_*)$ associated with the finite sequence $t_1-1,\ldots,t_n-1\in R$ is the $\IZ$-graded $R$-complex given by
\[
\begin{array}{lr}
G_p:=\Lambda^p(R^n), & \\
g_p:G_p\to G_{p-1},& g_p(e_\mu):=\sum\limits_{k=1}^n (-1)^k(t_k-1)e_{\mu\setminus (\mu_k)},
\end{array}
\]
see also \cite[Section 4.5]{Weibel1994}. For fixed $q\in\left\lbrace0,1\right\rbrace$, $\alpha$ induces an $R$-module structure on $K_q(A)$. We pass to the corresponding cohomological Koszul complex  $(G^*,g^*)$ with coefficients in the $R$-module $K_q(A)$
\[\def\arraystretch{1.25}
\begin{array}{l}
G^p:=\op{Hom}_R(G_p,K_q(A)),\\
g^p:G^p\to G^{p+1},\quad g^p(f):=f\circ g_{p+1}.
\end{array}\]
Using the $R$-module isomorphism
\[
G^p\cong \op{Hom}_R(G_p,R)\otimes_R K_q(A)\cong G_p\otimes_R K_q(A)\cong C_{PV}^{p,q},
\]
one can check that the two complexes $(C_{PV}^{*,q},d_{PV}^{*,q})$ and $(G^*,g^*)$ are isomorphic, and hence give rise to the same cohomology groups.

\begin{cor}
Let $A$ be a \cstar-algebra and $\alpha:\IZ^n\curvearrowright A$ an action. Let $(E_k,d_k)_{k\geq1}$ denote the Baum-Connes spectral sequence for $\alpha$. Then the $E_2$-term satisfies
\[
E_2^{p,q}\cong H^p(\IZ^n,K_q(A)),\quad p,q\in\IZ.
\] 
\end{cor}
\begin{proof}
By the definition of group cohomology, we only have to show that the Koszul complex $(G_*,g_*)$ defines a projective $R$-resolution of $\IZ$ (regarded as a trivial module over $R$). For this, it is sufficient to know that the finite sequence $t_1-1,\ldots,t_n-1$ is \emph{regular}, that is, $(t_1-1,\ldots,t_n-1)R\neq R$, and for $i=1,\ldots,n$, the element $t_i-1$ defines a non-zero-divisor in $R/(t_1-1,\ldots,t_{i-1}-1)R$, see \cite[Chapter XXI, Theorem 4.6a)]{Lang1997}. The first condition holds since $(t_1-1,\ldots,t_n-1)R$ is the augmentation ideal, which satisfies
\[
R/(t_1-1,\ldots,t_n-1)R\cong \IZ.
\]
Concerning the second condition, note that for $i=1,\ldots,n$, there is an isomorphism
\[
R/(t_1-1,\ldots,t_{i-1}-1)R\cong \IZ[\IZ^{n-i+1}].
\]
However, for $k\in \IN$, the group ring $\IZ[\IZ^k]$ is known to have no zero-divisors, see \cite{Malcev1948, Neumann1949}.
\end{proof}

Although Theorem \ref{partialDescriptionDifferentials} is very useful, it does not provide a complete description of the differentials $d_k$. First, not every element $g\in E_k^{p,q}$ is decomposable  in the sense that there are $x_\mu\in K_{q}(A)$, indexed by $\mu\in T(p,n)$, such that $[x\otimes e_\mu]\in E_k^{p,q}$ and
\[
g=\sum\limits_{\mu\in T(p,n)}[x_\mu\otimes e_\mu]\in E_k^{p,q}.
\]
However, even if $x\in K_q(A)$ satisfies $[x\otimes e_\mu]\in E_k^{p,q}$ for some $\mu\in T(p,n)$, it is not clear whether $x$ defines an element $[x]\in E(\mu^\bot)_k^{0,q}$. In fact, we do not automatically obtain a lift for the corresponding element $y\in K_{p+q}(S^pA)$ to an element in $K_{p+q}(S^p F(\mu^\bot)_{k-1})$ if we know that $y\otimes e_\mu\in K_{p+q}(I_p)$ lifts to an element in $K_{p+q}(F_{p+k-1})$. Nevertheless, this second problem does not occur for $k=2$.

\begin{cor}
\label{differentialsE2}
Let $0\leq p\leq n-2$ and $\mu\in T(p,n)$. Assume that $x\in K_q(A)$ gives rise to an element $[x\otimes e_\mu]\in E_{2}^{p,q}$. Then $[x]\in E(\mu^\bot)_2^{0,q}$, and with the notation from Theorem \ref{partialDescriptionDifferentials}, it follows that
\[
d_{2}^{p,q}([x\otimes e_\mu])=\left[\sum\limits_{\lambda\in T(2,n)}\rho(\lambda)_q(y_\lambda)\otimes(e_\mu\wedge e_\lambda)\right]\in E_2^{p+2,q-1}.
\]
\end{cor}
\begin{proof}
Using the notation of the proof of Theorem \ref{partialDescriptionDifferentials}, we consider the commutative diagram with exact rows
\[
\xymatrix{
0 \ar[r] & S^p(SA)^{\binom{n-p}{1}} \ar[r] \ar[d]^/-0.1cm/{\eta} & S^pF(\mu^\bot)_1 \ar[r] \ar[d] & S^p A \ar[r] \ar[d] & 0\\
0 \ar[r] & (S^{p+1}A)^{\binom{n}{p+1}} \ar[r] & F_{p+1} \ar[r] & F_p \ar[r] & 0
	}
\]
Naturality of $K$-theory yields a commutative diagram
\[
\xymatrix@C+0.2cm@R+0.1cm{
K_{p+q}(S^pA) \ar[r]^/-0.9cm/{\tilde{\rho}_{p+q}} \ar[d] & K_{p+q+1}(S^{p+1}A)\otimes_\IZ \Lambda^1(\IZ^{n-p})\ar[d]^{K_{p+q+1}(\eta)}\\
K_{p+q}(F_p) \ar[r]^/-0.9cm/{\rho_{p+q}} & K_{p+q+1}(S^{p+1}A)\otimes_\IZ \Lambda^{p+1}(\IZ^n)
	}
\]
Let $w\in K_{p+q}(S^pA)$ be the unique element corresponding to $x\in K_q(A)$ under the Bott isomorphism. By the definition of the $E_2$-term, we have that $\rho_{p+q}(w\otimes e_\mu)=0$. The proof of Theorem \ref{partialDescriptionDifferentials} shows that $K_*(\eta)$ is injective, and hence $\tilde{\rho}_{p+q}(w)=0$ as well. Again by Bott periodicity, this gives rise to a lift $y\in K_q(F(\mu^{\bot})_1)$ for $x\in K_q(A)$ under the map induced by the surjection $F(\mu^\bot)_1\to F(\mu^\bot)_0=A$. Hence, $[x]\in E(\mu^\bot)_2^{p,q}$, and the claim follows from Theorem \ref{partialDescriptionDifferentials}.
\end{proof}

%%%%%%%%%%%%%%%%%%%%%%%%%%%%%%%%%%%%%%%%%%%%%%%%%%%%%%%%%%%%%%%%%%%%%%%%%%%%%%%%%%%%%%%%%%%%

\section{Lifts under the boundary map of the Pimsner-Voiculescu sequence}
\label{Section Bott}
\noindent
In this section, we provide concrete lifts for elements in the image of the boundary map $\rho_*: K_*(A\rtimes_\alpha \IZ)\to K_{*+1}(A)$ of the Pimsner-Voiculescu sequence. These lifts will be essential for the concrete description of the second page differential in the next section. We first recall Exel's \cite{Exel1993} definition and some basic properties of Bott elements associated with almost commuting unitaries.

Consider the following extension of \cstar-algebras
\[
\xymatrix{
0 \ar[r] & \CK\otimes \CC(\IT) \ar[r] & \CT\otimes \CC(\IT) \ar[r] & \CC(\IT)\otimes \CC(\IT) \ar[r] & 0
}
\]
induced by the canonical surjection $\CT\to \CC(\IT)$. Up to a sign, the Bott element $\Fb\in K_0(\CC(\IT^2))$ is characterised by the property that its image under the corresponding index map $\rho_0:K_0(\CC(\IT)\otimes \CC(\IT))\to K_1(\CC(\IT))$ is a generator for $K_1(\CC(\IT))$. We fix the convention that $\rho_0(\Fb) = [z]$.
 
For $\eps\geq 0$, Exel \cite{Exel1993} defines the \emph{soft torus} $A_\eps$  as the universal \cstar-algebra
\[
A_\eps:=\mathrm{C}^*(u_\eps,v_\eps\ \text{unitaries}\ :\ \left\Vert\left[u_\eps,v_\eps\right]\right\Vert\leq\eps).
\]
It is obvious from the definition that $A_0=\CC(\IT^2)$, and that for $\eps\geq2$ the soft torus $A_\eps$ coincides with the full group \cstar-algebra of the free group in two generators. There is a canonical surjective $*$-homomorphism
\[
\varphi_\eps: A_\eps\to \CC(\IT^2)\quad \text{with}\ \varphi_\eps(u_\eps):= z_1, \ \varphi_\eps(v_\eps):=z_2.
\]
By \cite[Theorem 2.4]{Exel1993}, $K_*(\varphi_\eps)$ is an isomorphism whenever $\eps<2$, and in this case we define 
\[
\Fb_\eps:=K_0(\varphi_\eps)^{-1}(\Fb)\in K_0(A_\eps).
\]

\begin{definition}[following \cite{Exel1993}]
Let $0<\eps<2$, $B$ a unital \cstar-algebra, and $u,v\in B$ unitaries satisfying $\Vert\left[u,v\right]\Vert\leq\eps$. The universal property of the soft torus $A_\eps$ yields a unique $*$-homomorphism $\varphi:A_\eps\to B$ with  $\varphi(u_\eps)=u$ and $\varphi(v_\eps)=v$. Then
\[
\kappa(u,v):=K_0(\varphi)(\Fb_\eps)\in K_0(B)
\]
is called the \emph{Bott element} associated with $u$ and $v$.
\end{definition}

Note that $\kappa(u,v)$ is independent of $\eps$ as long as $\Vert\left[u,v\right]\Vert\leq\eps$. By definition, $\kappa(z_1,z_2)=\Fb\in K_0(\CC(\IT^2))$. If $\varphi:A\to B$ is a unital $*$-homomorphism and $u,v\in A$ are unitaries with $\Vert\left[u,v\right]\Vert<2$, then 
\[
K_0(\varphi)(\kappa(u,v))=\kappa(\varphi(u),\varphi(v)).
\]

For small tolerance $\eps>0$, the Bott element $\kappa(u,v)$ is given (up to a sign) by the following description due to Loring \cite{Loring1988}. Consider the real-valued functions $f,g,h\in \CC(\IT)$ defined as
\[
\def\arraystretch{3}
\begin{array}{lcl}
f(e^{2\pi it}) & = &
\begin{cases}
1-2t & \text{, \quad if } 0\leq t\leq 1/2,\\
-1+2t & \text{, \quad if } 1/2\leq t\leq 1,
\end{cases}\\
g(e^{2\pi it}) & = &
\begin{cases}
(f(\exp(2\pi it))-f(\exp(2\pi it))^2)^{1/2} & \text{, \quad if } 0\leq t\leq 1/2,\\
0 & \text{, \quad if } 1/2\leq t\leq 1,
\end{cases}\\
h(e^{2\pi it}) & = &
\begin{cases}
0 & \text{, \quad if } 0\leq t\leq 1/2,\\
(f(\exp(2\pi it))-f(\exp(2\pi it))^2)^{1/2} & \text{, \quad if } 1/2\leq t\leq 1,
\end{cases}
\end{array}\]
and set
\[
e(u,v):=\begin{pmatrix}f(v)&g(v)+h(v)u\\g(v)+u^{*}h(v)&1-f(v)\end{pmatrix}\in M_{2}(B).
\]
Clearly, $e(u,v)$ is self-adjoint for any choice of unitaries $u,v$. Moreover, direct calculations show that $e(u,v)$ is a projection whenever $u$ and $v$ commute. Loring observed in \cite[Proposition 3.5]{Loring1988} that there is a universal constant $\delta>0$ such that whenever $\Vert\left[u,v\right]\Vert<\delta$, then the spectrum of $e(u,v)$ does not contain $1/2$. In this case, $\chi_{[1/2,\infty)}(e(u,v))\in M_{2}(B)$ is a projection, and Loring's Bott element is given as
\[
[\chi_{[1/2,\infty)}(e(u,v))]-[1]\in K_0(B).
\]

The Bott elements also have the following well-known properties. We leave the proof to the reader.

\begin{prop}
\label{propertiesBott0}
Let $B$ be a unital \cstar-algebra and $u,v,u_1,v_1,\ldots,u_n,v_n\in B$ unitaries. Then the following statements hold true:
\begin{enumerate}[i)]
\item If $u_t\in B$ is a homotopy of unitaries with $\Vert\left[u_t,v\right]\Vert<2$ for all $t\in[0,1]$, then $\kappa(u_0,v)=\kappa(u_1,v)$.
\item If $\Vert\left[u_i,v_i\right]\Vert<2$ for $i=1,\ldots,n$, then
\[	
\kappa(\diag(u_1,\ldots, u_n),\diag(v_1,\ldots ,v_n))=\sum^n_{i=1}\kappa(u_i,v_i).
\]
\item If $\sum\limits_{i=1}^n\Vert\left[u,v_i\right]\Vert<2$, then $\kappa(u,v_1v_2\ldots v_n)=\sum\limits^n_{i=1}\kappa(u,v_i)$.
\item If $\Vert\left[u,v\right]\Vert<2$, then $\kappa(u,v)=-\kappa(u,v^*)=-\kappa(v,u)$.
\end{enumerate}
\end{prop}

For a unital, purely infinite and simple \cstar-algebra $A$, Elliott and R\o rdam showed in \cite[Theorem 2.2.1]{ElliottRordam1995} that every element $x\in K_0(A)$ is a Bott element $x=\kappa(u,v)$ for some pair of commuting unitaries $u,v\in A$ (with full spectrum). On the other hand, if $A$ is a unital \cstar-algebra admitting a tracial state $\tau$, then every Bott elements in $K_0(A)$ associated with exactly commuting unitaries vanishes under $K_0(\tau)$. In fact, if $\tau_2$ denotes the induced (unnormalized) trace on $M_2(A)$, then $\tau_2(e(u,v))=1$.

It will turn out to be convenient to consider the following analogous notion of Bott elements in the $K_1$-group of a unital \cstar-algebra.
\begin{nota}
Let $A$ be a unital \cstar-algebra. Let $p\in A$ be a projection and $u\in A$ a unitary commuting with $p$. Then $pup+1-p\in A$ is a unitary, and we define the \emph{Bott element} associated with $p$ and $u$ as
\[
\kappa(p,u):=[pup+1-p]\in K_1(A).
\]
Observe that the  Bott isomorphism $K_0(A)\stackrel{\cong}{\longrightarrow} K_1(SA)$ indeed sends $[p]$ to $\kappa(p,z)$.
\end{nota}

Recall the \emph{Pimsner-Voiculescu exact sequence} \cite{PimsnerVoiculescu1980}
\[
\xymatrix@C+0.4cm@R+0.1cm{
K_{0}(A) \ar[r]^/-0.5em/{K_0(\alpha)-\id} & K_0(A) \ar[r]^/-0.3cm/{K_0(j)} & K_{0}(A\rtimes_\alpha\IZ) \ar[d]^{\rho_{0}}\\					
K_{1}(A\rtimes_\alpha\IZ) \ar[u]^{\rho_{1}} & K_1(A) \ar[l]^/-0.3cm/{K_1(j)} & K_1(A) \ar[l]^/-0.5em/{K_1(\alpha)-\id}
}
\]
with $j: A\into A\rtimes_\alpha\IZ$ denoting the canonical embedding. The six-term exact sequence associated with the \emph{Toeplitz extension}
\[
\xymatrix{
0 \ar[r]& \CK\otimes A \ar[r] & \CT(A,\alpha) \ar[r] &  A\rtimes_\alpha\IZ \ar[r] & 0
} 
\]
serves as a starting point for the proof of the Pimsner-Voiculescu sequence, where
\[
\CT(A,\alpha):=\mathrm{C}^*(1\otimes a,\ v\otimes u\ :\ a\in A)\subset \CT\otimes (A\rtimes_\alpha\IZ)
\]
is the \emph{crossed Toeplitz-algebra} associated to $\alpha$, see also \cite{Cuntz1984}. Here and in the following, if not stated otherwise, $u\in A\rtimes_\alpha \IZ$ denotes the canonical unitary implementing $\alpha$. It is important for us to observe that the boundary maps of the Pimsner-Voiculescu sequence coincide (at least up to an application of the stabilisation isomorphism) with the ones of the six-term exact sequence associated with the Toeplitz extension.

The Pimsner-Voiculescu sequence is natural in the sense that given an equivariant $*$-homomorphism $\varphi:(A,\alpha,\IZ)\to(B,\beta,\IZ)$, the following diagram commutes
\[
\xymatrix{
K_*(A) \ar[rr]^{K_*(\alpha)-\id}\ar[d]^{K_*(\varphi)} & & K_*(A) \ar[r]\ar[d]^{K_*(\varphi)} & K_*(A\rtimes_{\alpha}\IZ)  \ar[d]^{K_*(\check{\varphi})} \ar[r]^{\rho_*} &  K_{*+1}(A) \ar[d]^{K_{*+1}(\varphi)} \\
K_*(B) \ar[rr]^{K_*(\beta)-\id} & & K_*(B) \ar[r] & K_*(B\rtimes_{\beta}\IZ) \ar[r]^{\rho_*} &  K_{*+1}(B)
}
\]

\begin{nota}
Let $A$ be a \cstar-algebra and $\alpha\in \Aut(A)$ an automorphism. For $n\in \IN$, we write $\alpha^{(n)} := \alpha\otimes \id \in \Aut(A\otimes M_n(\IC))$. Similarly, we define $a^{(n)}:=a\otimes 1_n\in A\otimes M_n(\IC)$ for a given element $a\in A$.
\end{nota}

Assume that $A$ is unital. We now describe preimages of the boundary map $\rho_1:K_1(A\rtimes_\alpha\IZ)\to K_0(A)$ of the Pimsner-Voiculescu sequence. Observe that every element $g\in K_0(A)$ can be expressed as $g=[p]-[1_n]$ for some $p\in \CP_m(A)$ and $n\geq0$. It is obvious that $g\in \ker(K_0(\alpha)-\id)$ if and only if $[p]\in \ker(K_0(\alpha)-\id)$. Hence, it suffices to describe lifts for elements of the form  $[p]\in\im(\rho_1)=\ker(K_0(\alpha)-\id)$.

\begin{prop}
\label{index1}
Let $A$ be a unital \cstar-algebra, $\alpha\in \Aut(A)$, and $p\in \CP_k(A)$ a projection satisfying $[p]\in \ker(K_0(\alpha)-\id)$. By the standard picture of $K_0(A)$, we find $l,m\geq0$ and a unitary $w\in\CU_n(A)$ such that 
\[
\alpha^{(n)}(q)=wqw^*,
\]
where $n:=k+l+m$ and $q:=\diag(p,1_l,0_m) \in \CP_n(A)$. Then
\[
\rho_1(\kappa(q,w^*u^{(n)})-[u^{(l)}])=[p].
\]
\end{prop}
\begin{proof}
Assume first that $k=1$ and $l=m=0$. It is easy to verify that 
\[
y:=v\otimes(pw^*up)+ 1\otimes(1-p)\in \CT(A,\alpha)
\]
is an isometry and a lift for $pwu^*p+1-p\in A\rtimes_{\alpha}\IZ$. Using the partial isometry picture of the index map, one computes
\[\def\arraystretch{1.3}
\begin{array}{lclll}
\rho_1(\kappa(p,w^*u)) & = & [1-yy^*]-[1-y^*y] & = & [1-yy^*]\\
 & = & [(1-vv^*)\otimes p]  & \hspace{-1cm}\in & \hspace{-1.2cm} K_0(\CK\otimes A).
\end{array}
\]
By the stabilisation isomorphism $K_0(A)\cong K_0(\CK\otimes A)$, we deduce that
\[
\rho_1(\kappa(p,w^{*}u))=[p]\in K_0(A).
\]

Now, let $q\in \CP_n(A)$ be as in the statement. The canonical isomorphism $\eta:M_n(A)\rtimes_{\alpha^{(n)}}\IZ\stackrel{\cong}{\longrightarrow} A\rtimes_\alpha \IZ \otimes M_n(\IC)$ fits into a commutative diagram
\[
\xymatrix@C+0.4cm@R+0.1cm{
K_*(M_n(A)\rtimes_{\alpha^{(n)}}\IZ) \ar[r]^/.4cm/{\rho^{(n)}_*} \ar[d]_{K_*(\eta)} & K_{*+1}(A) \\
K_*(A\rtimes_{\alpha}\IZ) \ar[ur]_{\rho_*} &
}
\]
relating the boundary maps of the respective Pimsner-Voiculescu sequences. Hence,
\[
\rho_1(\kappa(q,w^*u^{(k)}))=(\rho_1\circ K_1(\eta))(\kappa(q,w^*u))=\rho^{(k)}_1(\kappa(q,w^*u))=[q].
\]
It follows that
\[
\rho_1(\kappa(q,w^*u^{(k)})-[u^l])=[q]-[1_l]=[\diag(p, 1_l)]-[1_l]=[p].
\]
\end{proof}

The lifts for the boundary map $\rho_0: K_0(A\rtimes_\alpha \IZ)\to K_1(A)$ require an alternative picture for $K_1(A)$. Using the natural identification $K_*(A)\cong KK(\CC(\IT),A)$, Dadarlat's result \cite[Theorem A]{Dadarlat1995} for $X=\IT$ admits the following characterisation.

\begin{theorem}[cf.~ \cite{Dadarlat1995}]
\label{DescriptionK1Dadarlat}
Let $A$ be a unital \cstar-algebra and $u,v\in \CU(A)$ two unitaries. Then $[u]=[v]\in K_1(A)$ if and only if for every $\eps>0$, there exist $k\geq1$, $\lambda_1,\ldots,\lambda_k\in \IT$, and a unitary $w\in M_{k+1}(A)$ such that
\[
\left\Vert w(\diag(u,\lambda_{1},\ldots,\lambda_{k}))w^{*} - \diag(v,\lambda_1,\ldots,\lambda_k) \right\Vert\leq\eps.
\]
\end{theorem}

For $\eps>0$ define the universal \cstar-algebra
\[
T_\eps:=\mathrm{C}^*\left(s\ \text{isometry}, u\ \text{unitary}\ :\ \Vert \left[s,u\right]\Vert\leq\eps,\ u(1-ss^*)=(1-ss^*)u\right).
\]
Consider the canonical surjection $\pi_\eps:T_\eps\to A_\eps$ given by $\pi_\eps(s)=u_\eps$ and $\pi_\eps(u)=v_\eps$. Observe that the surjective $*$-homomorphism $\psi_\eps:T_\eps\to \CT\otimes \CC(\IT)$ given by $\psi_\eps(s)=v$ and $\psi_\eps(u)=z$ fits into the following commutative diagram
\[
\xymatrix{
0 \ar[r]& \CK\otimes\CC(\IT) \ar[r] \ar@{=}[d] & T_{\eps}\ar[r]^{\pi_\eps}\ar[d]^{\psi_{\eps}} & A_{\eps} 		 	\ar[r]\ar[d]^{\varphi_{\eps}}& 0\\
0 \ar[r] & \CK\otimes\CC(\IT) \ar[r] & \CT\otimes\CC(\IT) \ar[r] & \CC(\IT)\otimes\CC(\IT) \ar[r]& 0
}
\]
Naturality of $K$-theory allows us to compare the occurring boundary maps
\[
\xymatrix{
K_0(A_\eps) \ar[r]^{\rho_\eps} \ar[d]_{K_0(\varphi_\eps)} & K_1(\CC(\IT))\\
K_0(\CC(\IT)\otimes\CC(\IT)) \ar[ru]_\rho &  	
}
\]
For $\eps<2$, we therefore get that $\rho_\eps(\Fb_\eps)=\rho(\Fb)=[z]\in K_1(\CC(\IT))$.

\begin{prop}
\label{index0}
Let $A$ be a unital \cstar-algebra, $\alpha\in \Aut(A)$, and $x\in \CU_k(A)$ a unitary satisfying $[x]\in \ker(K_1(\alpha)-\id)$. An application of Theorem \ref{DescriptionK1Dadarlat} yields $l\geq0$, $\lambda_1,\ldots,\lambda_l\in\IT$, and $w\in\CU_m(A)$ satisfying
\[
\left\Vert\alpha^{(m)}(y)-wyw^*\right\Vert < 2,
\]
where $m:=k+l$ and $y:=\diag(x,\lambda_1,\cdots,\lambda_l)\in \CU_m(A)$. Then
\[
\rho_0(\kappa(w^*u^{(m)},y))=[x].
\]
\end{prop}
\begin{proof}
First assume that $k=1$ and $l=0$. For suitably chosen $\eps<2$, there is a $*$-homomorphism
\[
\psi:T_\eps\to \CT(A,\alpha),\quad \psi(s)=v\otimes w^*u,\ \psi(u)=1\otimes x.
\]
This homomorphism fits into the commutative diagram
\[
\xymatrix{
0 \ar[r]& \CK\otimes \CC(\IT) \ar[r] \ar[d]^{\id_\CK\otimes \nu}& T_\eps \ar[r] \ar[d]^\psi & A_\eps \ar[r] \ar[d]^\varphi & 0\\
0 \ar[r]& \CK\otimes A \ar[r] & \CT(A,\alpha) \ar[r] &	A\rtimes_\alpha\IZ \ar[r]& 0
}
\]
with $\varphi$ and $\nu$ given by $\varphi(u_\eps)=w^{*}u$, $\varphi(v_\eps)=x$, and $\nu(z)=x$, respectively. By stability of $K$-theory,
\[\def\arraystretch{1.5}
\begin{array}{lclcl}
\rho_0(\kappa(w^*u,x)) & = & (\rho_0\circ K_0(\varphi))(\Fb_\eps)  & = & (K_1(\id_\CK\otimes \nu)\circ\rho_\eps)(\Fb_\eps)\\
 & = & K_1(\nu)([z]) & = & [x]\hspace{0.2cm}\in\hspace{0.2cm} K_1(A).
\end{array}
\]

If $y\in\CU_m(A)$ is as in the statement, then by a similar reasoning as in the proof of Proposition \ref{index1},
\[
\rho_0(\kappa(w^* u^{(m)},y))=(\rho_0\circ K_0(\eta))(\kappa(w^*u,y))=\rho^{(m)}_{0}(\kappa(w^*u,y))=[y]=[x].
\]
\end{proof}

%%%%%%%%%%%%%%%%%%%%%%%%%%%%%%%%%%%%%%%%%%%%%%%%%%%%%%%%%%%%%%%%%%%%%%%%%%%%%%%%%%%%%%%%%%%%

\section{The second page differential associated with a $\IZ^2$-action}
\noindent
In this section, we provide an explicit description of the second page differential of the spectral sequence associated with a $\IZ^2$-action in terms of Bott elements. For this, we make use of the concrete lifts for the boundary map of the Pimsner-Voiculescu sequence described in Section \ref{Section Bott}.

\begin{nota}
Let $A$ be a \cstar-algebra and $\alpha:\IZ^2\curvearrowright A$ an action. 
\begin{itemize}
\item We denote by $\tilde{\alpha}_2\in \Aut(\CM_{\alpha_1}(A))$ the $*$-automorphism given by
\[
\tilde{\alpha}_2(f)(t):=\alpha_2(f(t)),\quad f\in \CM_{\alpha_1}(A),\ t\in [0,1]^n.
\]
\item For $i=1,2$, we write  $d_*(\alpha_i) := K_*(\alpha_i)-\id$ and denote by
\[
\begin{array}{c}
\op{k}(d_*(\alpha_2)): \kernel(d_*(\alpha_1))\to \kernel(d_*(\alpha_1)),\\
\op{co}(d_*(\alpha_2)):\op{coker}(d_*(\alpha_1))\to \op{coker}(d_*(\alpha_1))
\end{array}
\]
the respective natural homomorphisms induced by $d_*(\alpha_2)$.
\end{itemize}
\end{nota}

Given a \cstar-algebra $A$ and an action $\alpha:\IZ^2\curvearrowright A$, the corresponding $E_2$-term is concentrated in $p=0,1,2$ and given as
\[
\def\arraystretch{1.5}
\begin{array}{lcl}
E^{0,q}_2 & = & \kernel(K_q(\alpha_1) - \id) \cap \kernel(K_q(\alpha_2) - \id), \\
E^{1,q}_{2} & = & \frac{\kernel((K_q(\alpha_2)-\id)\oplus (K_q(\alpha_1)-\id))}{\op{im}((K_q(\alpha_1)-\id,\ K_q(\alpha_2)-\id))},\\
E^{2,q}_2  & = & K_q(A) \big / \left\langle \op{im}(K_q(\alpha_1) - \id),\op{im}(K_q(\alpha_2) - \id) \right\rangle.
\end{array}
\]
Moreover, the $E_\infty$-term coincides with the $E_3$-term. Hence, up to group extension problems, $K_*(A\rtimes_\alpha\IZ^2)$ is uniquely determined by the induced action of $\alpha$ on $K$-theory and $d_2$.

Since $d_2:E_2\to E_2$ has bidegree $(2,-1)$, it reduces to
\[
d^{0,q}_2:E^{0,q}_2\to E^{2,q-1}_2,\quad q=0,1.
\]
Consider the following commutative diagram with exact rows
\[
\xymatrix{
0 \ar[r] & SA \ar[r] \ar[d]^{S\alpha_2} & \CM_{\alpha_1}(A) \ar[r]^/0.15cm/{\op{ev}_0} \ar[d]^{\tilde{\alpha}_2} & A \ar[r] \ar[d]^{\alpha_2}& 0 \\
0 \ar[r] & SA \ar[r] & \CM_{\alpha_2}(A) \ar[r]^/0.15cm/{\op{ev}_0} & A \ar[r] & 0
}
\]
and the induced diagram, which is obtained by passing to the respective six-term exact sequences in $K$-theory (and applying Bott periodicity)
\[
\xymatrix@C+0.25cm{
0 \ar[r] & \coker(d_{*-1}(\alpha_1)) \ar[r] \ar[d]^{\op{co}(d_{*-1}(\alpha_2))} & K_*(\CM_{\alpha_{1}}(A)) \ar[r]^{K_*(\op{ev}_0)} \ar[d]^{d_*(\tilde{\alpha}_2)} & \ker(d_*(\alpha_1)) \ar[d]^{\op{k}(d_*(\alpha_2))} \ar[r]& 0\\
0 \ar[r] & \coker(d_{*-1}(\alpha_1)) \ar[r] & K_*(\CM_{\alpha_1}(A)) \ar[r]^{K_*(\op{ev}_0)} & \ker(d_*(\alpha_1)) \ar[r] & 0
}
\]
By applying the Snake Lemma (see for example \cite[1.3.2]{Weibel1994}) to this diagram, we obtain a group homomorphism
\[
d_*(\alpha):E^{0,*}_2\to E^{2,*-1}_2,
\]
which, as the next result shows, coincides with $d_2^{0,*}$.

\begin{prop}
\label{differentialObstructionHomomorphism}
Let $A$ be a \cstar-algebra and $\alpha:\IZ^2\curvearrowright A$ an action. Then the associated second page differential $d_2$ satisfies $d_2^{0,q}=d_q(\alpha)$ for $q=0,1$.
\end{prop}
\begin{proof}
Recall the mapping torus cofiltration \eqref{cofiltrationMappingTorus} associated with $\alpha$
\[
\xymatrix{
	\CM_\alpha(A) \ar@{->>}[r]^/0.1cm/{\pi_2} & F_1 \ar@{->>}[r]^{\pi_1} & A \ar[r] & 0.
	}
\]
In this case, $F_1$ is given as the pullback of $\CM_{\alpha_1}(A)$ and $\CM_{\alpha_2}(A)$ along the respective evaluations at $0$. Consider the commutative diagram with exact rows
\[
\xymatrix{
0 \ar[r] & S^2 A \ar[r] \ar@{^(->}[d] & \CM_\alpha(A) \ar[r]^/0.1cm/{\pi_2} \ar@{=}[d] & F_1 \ar[r] \ar@{->>}[d]^/-0.1cm/{\pi} & 0\\
0 \ar[r] & S\CM_{\alpha_1}(A) \ar[r] & \CM_\alpha(A) \ar[r]^/-0.1cm/{\op{ev}_0} & \CM_{\alpha_1}(A) \ar[r] & 0 	
	}
\]
Naturality of $K$-theory and Bott periodicity give rise to a commutative diagram
\[
	\xymatrix{
K_{q}(F_1)\ar[r]^/-0.2cm/{\rho_q} \ar[d]^{K_{q}(\pi)} & K_{q-1}(S^2 A)\ar[r]^{\cong} \ar[d] & K_{q-1}(A) \ar@{->>}[r] \ar[d] &	 E_2^{2,q-1}\\
K_{q}(\CM_{\alpha_1}(A)) \ar[r]^/-0.2cm/{\rho_q} \ar@/_.7cm/[rr]_{K_{q}(\tilde{\alpha}_2)-\id} & K_{q-1}(S\CM_{\alpha_1}(A)) \ar[r]^/0.1cm/{\cong}  & K_q(\CM_{\alpha_1}(A)) & 
	}
\]
Denote by $d:K_q(F_1)\to E_2^{2,q-1}$ the map that is obtained by following the upper row of the last diagram. Let $x\in E_2^{0,q} = \op{im}(K_q(\pi_1))\subseteq K_q(A)$ be given and take some $\bar{x}\in K_q(F_1)$ with $K_q(\pi_1)(\bar{x})=x$. By the definition of the second page differential, $d_2^{0,q}(x)=d(\bar{x})$. On the other hand, it is clear that $K_q(\pi)(\bar{x})$ defines a lift for $x$ under $K_q(\op{ev_0}):K_q(\CM_{\alpha_1}(A))\to K_q(A)$. By the definition of the Snake Lemma homomorphism, we get that $d_q(\alpha)(x)=d(\bar{x})$. This concludes the proof.
\end{proof}

The following result now follows easily from Corollary \ref{differentialsE2}.

\begin{cor}
\label{completeDescriptiond2}
Let $A$ be a \cstar-algebra and $\alpha:\IZ^n\curvearrowright A$ an action with the property that $K_*(\alpha_i)=\id$ for $i=1,\ldots,n$. Then
\[
d_2^{p,q}(x\otimes e)=\sum_{\mu\in T(2,n)} d_{q}(\alpha(\mu))(x)\otimes (e\wedge e_\mu)
\]
for every $x\otimes e \in E_1^{p,q}=E_2^{p,q}$. 
\end{cor}

The isomorphism $K_*(A\rtimes_\alpha \IZ)\cong K_{*+1}(\CM_\alpha(A))$ described by Paschke \cite{Paschke1983} intertwines the automorphisms $K_*(\check{\alpha}_2)$ and $K_{*+1}(\tilde{\alpha}_2)$, see also \cite[Theorem 1.2.6]{Barlak2014}. Therefore, Proposition \ref{differentialObstructionHomomorphism} yields that $d_2^{0,*+1}$ can be regarded as the Snake Lemma homomorphism of
\[
\xymatrix{
0 \ar[r] & \coker(d_*(\alpha_1)) \ar[r] \ar[d]^{\op{co}(d_*(\alpha_2))}& K_*(A\rtimes_{\alpha_1}\IZ) \ar[r]^{\rho_*} \ar[d]^{d_*(\check{\alpha}_2)}& \ker(d_{*+1}(\alpha_1)) \ar[d]^{\op{k}(d_{*+1}(\alpha_2))}\ar[r]& 0\\
0 \ar[r] & \coker(d_*(\alpha_1)) \ar[r] & K_*(A\rtimes_{\alpha_{1}}\IZ) \ar[r]^{\rho_*} & \ker(d_{*+1}(\alpha_1)) \ar[r]& 0\\
}
\]
which is induced by the naturality of the Pimsner-Voiculescu sequence applied to $\alpha_2:(A,\alpha_1,\IZ)\stackrel{\cong}{\longrightarrow}(A,\alpha_1,\IZ)$. This characterisation and Proposition \ref{index1} now permit the following description of $d_2^{0,0}$, provided that $A$ is unital.

\begin{theorem}
\label{obstructiond0}
Let $A$ be a unital \cstar-algebra and $\alpha:\IZ^2\curvearrowright A$ an action. Let $p\in \CP_k(A)$ be a projection satisfying $[p]\in E_2^{0,0}$. Find $l,m\geq0$ and unitaries $v,w\in \CU_n(A)$ with
\[
\alpha^{(n)}_1(q) = vqv^*\quad\text{and}\quad \alpha^{(n)}_2(q) = wqw^*,
\]
where $n:=k+l+m$ and $q:=\diag(p,1_l,0_m)\in\CP_n(A)$. Then
\[
d_2^{0,0}([p])=\left[\kappa(q,w^*\alpha^{(n)}_2(v)^*\alpha^{(n)}_1(w)v)\right]\in E_2^{2,-1}.
\]
\end{theorem}
\begin{proof}
If $\rho_1:K_1(A\rtimes_{\alpha_1}\IZ)\to K_0(A)$ denotes the index map of the Pimsner-Voiculescu sequence for $\alpha_1$, then Proposition \ref{index1} yields
\[
\rho_1(\kappa(q,v^*u^{(n)})-[u^{(l)}])=[p].
\]
One computes that
\[
\def\arraystretch{1.8}
\begin{array}{lll}
(K_0(\check{\alpha}_2)-\id)(\kappa(q,v^*u^{(n)})-[u^{(l)}])  & &\\
 & \hspace{-0.9cm} = & \hspace{-0.7cm} (K_0(\check{\alpha}_2)-\id)(\kappa(q,v^*u^{(n)}))\\
 & \hspace{-0.9cm} = & \hspace{-0.7cm} \kappa(\alpha^{(n)}_2(q),\alpha^{(n)}_2(v)^*u^{(n)}) - \kappa(q,v^*u^{(n)})\\
 & \hspace{-0.9cm} = & \hspace{-0.7cm} \kappa(wqw^*,\alpha^{(n)}_2(v)^*u^{(n)})+\kappa(q,{u^{(n)}}^*v)\\
 & \hspace{-0.9cm} = & \hspace{-0.7cm} \kappa(q,w^*\alpha^{(n)}_2(v)^*u^{(n)}w)+\kappa(q,{u^{(n)}}^*v)\\
 & \hspace{-0.9cm} = & \hspace{-0.7cm} \kappa(q,w^*\alpha^{(n)}_2(v)^*u^{(n)}w{u^{(n)}}^*v)\\
 & \hspace{-0.9cm} = & \hspace{-0.7cm}  \kappa(q,w^*\alpha^{(n)}_2(v)^*\alpha^{(n)}_1(w)v)\in K_1(A).
\end{array}
\]
By the definition of the Snake Lemma homomorphism, we get that 
\[
d_0(\alpha)([p])=\left[\kappa(q,w^*\alpha^{(n)}_2(v)^*\alpha^{(n)}_1(w)v)\right]\in E_2^{2,-1}.
\]
The proof now follows from Proposition \ref{differentialObstructionHomomorphism}.
\end{proof}

Observe that $d_2^{0,0}$ is completely determined by Theorem \ref{obstructiond0}. In fact, given $g\in E_2^{0,0}$, there is a projection $p\in \CP_m(A)$ and some $n\geq 0$ such that $g=[p]-[1_n]\in K_0(A)$. In this situation, $[p]\in E_2^{0,0}$ and $d_2^{0,0}(g)=d_2^{0,0}([p])$.

For the description of $d_2^{0,1}:E_2^{0,1}\to E_2^{2,0}$, we need the following perturbation result.

\begin{lemma}
\label{smallHomotopiesUnitaries}
Let $0<\eps< \frac{2}{3}$. Let $A$ be unital \cstar-algebra and $u,\bar{u},v\in \CU(A)$ unitaries satisfying 
\[
\Vert u-\bar{u}\Vert,\ \Vert \left[u,v\right] \Vert\leq\eps.
\]
Then there is a homotopy $u_t\in \CU(A)$ with $u_0=u$, $u_1=\bar{u}$, and $\Vert \left[u_t,v\right]\Vert\leq3\eps$ for all $t\in[0,1]$.
\end{lemma}
\begin{proof}
Since $\Vert u-\bar{u}\Vert<\frac{2}{3}$, the spectrum of $u^*\bar{u}$ does not contain $-1$. Therefore, we can define $h:=-i\log(u^*\bar{u})\in A$, where $\log$ denotes the principal branch of the logarithm. This yields a continuous path of unitaries $u_t:=u\exp(ith)\in A$, $t\in [0,1]$, with $u_0=u$ and $u_1=u\exp(\log(u^*\bar{u}))=\bar{u}$. For $s,t\in [0,1]$,
\[
\Vert u_s-u_t\Vert = \Vert 1-\exp(i(s-t)h)\Vert \leq  \Vert 1-\exp(ih)\Vert = \Vert u-\bar{u}\Vert \leq \eps.
\]
One now computes
\[
\Vert \left[u_t,v\right] \Vert \leq \Vert u_tv-uv\Vert+\Vert vu_t-vu\Vert+\Vert \left[u,v\right]\Vert \leq 3\eps.
\]
\end{proof}

\begin{theorem}
\label{obstructiond1}
Let $A$ be a unital \cstar-algebra and $\alpha:\IZ^2\curvearrowright A$ an action. Let $v\in \CU_k(A)$ be a unitary satisfying $[v]\in E_2^{0,1}$. By Theorem \ref{DescriptionK1Dadarlat}, there are $l\geq0$, $\lambda_1,\ldots,\lambda_l\in \IT$, and unitaries $x,y\in \CU_m(A)$ such that
\[
\left\Vert\alpha^{(m)}_1(w)-xwx^*\right\Vert,\ \left\Vert\alpha^{(m)}_2(w)-ywy^*\right\Vert < \frac{1}{2},
\]
where $m:=k+l$ and $w:=\diag(v,\lambda_1,\ldots,\lambda_l)\in \CU_m(A)$. Then
\[
d_2^{0,1}([v])=\left[\kappa(y^*\alpha^{(m)}_2(x)^*\alpha^{(m)}_1(y)x,w)\right]\in E_2^{2,0}.
\]
\end{theorem}
\begin{proof}
Using the isomorphism $M_m(A\rtimes_{\alpha_1}\IZ)\cong M_m(A)\rtimes_{\alpha_1^{(m)}}\IZ$, we compute that
\[
\left\Vert\left[x^*u^{(m)},w\right]\right\Vert = \left\Vert u^{(m)}w {u^{(m)}}^*-xw x^*\right\Vert = \left\Vert \alpha_1^{(m)}(w)-xw x^*\right\Vert\\
  <  \frac{1}{2}.
\]
By Proposition \ref{index0}, the boundary map $\rho_0:K_0(A\rtimes_{\alpha_1}\IZ)\to K_1(A)$ of the Pimsner-Voiculescu sequence for $\alpha_1$ satisfies
\[
\rho_0(\kappa(x^*u^{(m)},w))=[v].
\]
The naturality of the Bott elements and part {\it iv)} of Proposition \ref{propertiesBott0} yield
\[
\def\arraystretch{1.5}
\begin{array}{lll}
(K_0(\check{\alpha}_{2})-\id)(\kappa(x^*u^{(m)},w)) & & \\
 & \hspace{-0.5cm} = & \hspace{-0.3cm} \kappa(\alpha^{(m)}_{2}(x)^*u^{(m)},\alpha^{(m)}_{2}(w))-\kappa(x^*u^{(m)},w)\\
& \hspace{-0.5cm} = & \hspace{-0.3cm} \kappa(\alpha^{(m)}_{2}(x)^*u^{(m)},\alpha^{(m)}_{2}(w))+\kappa({u^{(m)}}^{*}x,w).
\end{array}
\]
Since 
\[
\left\Vert\alpha^{(m)}_2(w)-ywy^*\right\Vert,\ \left\Vert\left[\alpha^{(m)}_2(x)^*u^{(m)},\alpha^{(m)}_2(w)\right]\right\Vert<\frac{1}{2},
\]
we can apply Lemma \ref{smallHomotopiesUnitaries} and find a homotopy $w_t\in \CU_m(A)$ between $\alpha^{(m)}_2(w)$ and $ywy^{*}$ such that
\[
\left\Vert\left[w_t,\alpha^{(m)}_2(x)^*u^{(m)}\right]\right\Vert<\frac{3}{2}\quad \text{for all}\ t\in[0,1].
\]
By part {\it i)} of Proposition \ref{propertiesBott0} and the naturality of the Bott elements, we obtain that
\[
\begin{array}{lcl}
\kappa(\alpha^{(m)}_{2}(x)^*u^{(m)},\alpha^{(m)}_{2}(w)) & = & \kappa(\alpha^{(m)}_{2}(x)^*u^{(m)},ywy^*) \\
 & = & \kappa(y^*\alpha^{(m)}_2(x)^*u^{(m)}y,w).
\end{array}
\]
Moreover,
\[
\left\Vert\left[{u^{(m)}}^*x,w\right]\right\Vert+\left\Vert\left[y^*\alpha^{(m)}_{2}(x)^*u^{(m)}y,w\right]\right\Vert<\frac{1}{2}+\frac{3}{2}=2,
\]
and therefore part {\it iii)} of Proposition \ref{propertiesBott0} yields
\[\def\arraystretch{1.5}
\begin{array}{lcl}
(K_0(\check{\alpha}_{2})-\id)(\kappa(x^*u^{(m)},w))
 & = & \kappa(y^*\alpha^{(m)}_2(x)^*u^{(m)}y,w)+\kappa({u^{(m)}}^*x,w)\\
 & = & \kappa(y^*\alpha^{(m)}_2(x)^*u^{(m)}y{u^{(m)}}^*x,w)\\
 & = & \kappa(y^*\alpha^{(m)}_2(x)^*\alpha^{(m)}_1(y)x,w)\in K_0(A).
\end{array}\]
By the definition of the Snake Lemma homomorphism, it follows that 
\[
d_1(\alpha)([v])=\left[\kappa(y^*\alpha^{(m)}_2(x)^*\alpha^{(m)}_1(y)x,w)\right]\in E_2^{2,0}.
\]
The proof now follows from Proposition \ref{differentialObstructionHomomorphism}.
\end{proof}

%%%%%%%%%%%%%%%%%%%%%%%%%%%%%%%%%%%%%%%%%%%%%%%%%%%%%%%%%%%%%%%%%%%%%%%%%%%%%%%%%%%%%%%%%%%%%%

\section{Locally $KK$-trivial $\IZ^2$-actions on Kirchberg algebras}
\noindent
Let $A$ be a \cstar-algebra and $\alpha:\IZ^2\curvearrowright A$ an action with the property that $\alpha_1$ is homotopic to $\id_A$ in $\Aut(A)$. Fix a homotopy $\beta_t\in \Aut(A)$ between $\beta_0=\alpha_1$ and $\beta_1=\id_A$ and consider the induced $*$-automorphism $\phi\in \Aut(A\otimes \CC(\IT))$ given by
\[
\phi(f)(\exp(2\pi it))=\left(\beta_t\circ\alpha_2\circ\beta^{-1}_t\right)(f(\exp(2\pi it))),\; f\in A\otimes \CC(\IT),\ t\in[0,1].
\]
Note that $\phi$ is well-defined since $\alpha_1$ and $\alpha_2$ commute. Obviously, $\phi$ restricts to an automorphism $\phi^\prime:SA\stackrel{\cong}{\longrightarrow} SA$ and fits into the following commutative diagram with split-exact rows
\begin{align}
\label{commutativeAlpha0}
\begin{xy}
\xymatrix@C+0.2cm{
0 \ar[r] & SA \ar[r] \ar[d]^{\phi^\prime} & A\otimes \CC(\IT) \ar[r]^/0.25cm/{\op{ev}_1} \ar[d]^{\phi} & A \ar[d]^{\alpha_2} \ar[r] \ar@/_0.5cm/[l]_j & 0 \\
0 \ar[r] & SA \ar[r] & A\otimes \CC(\IT) \ar[r]^/0.25cm/{\op{ev}_1} & A \ar[r] \ar@/^0.5cm/[l]^j & 0
	}
\end{xy}
\end{align}
Here, $j:A \into A\otimes \CC(\IT)$ denotes the canonical embedding. Futhermore, there is a $*$-iso\-mor\-phism $\psi:\CM_{\alpha_1}(A)\stackrel{\cong}{\longrightarrow} A\otimes \CC(\IT)$ satisfying
\[
\psi(f)(\exp(2\pi it))=(\beta_t\circ\alpha_1^{-1})(f(t)),\quad f\in \CM_{\alpha_1}(a),\ t\in [0,1].
\]
Its restriction $\psi^{\prime}\in \Aut(SA)$ is homotopic to $\id_{SA}$ via $\psi^{\prime}_{s}\in \Aut(SA)$ given by
\[
\psi^{\prime}_{s}(f)(t)=(\beta_{st}\circ\alpha_{1}^{-1})(f(t)),\quad f\in SA,\ s,t\in[0,1].
\]

\begin{prop}
\label{ObstructionMapCircle}
Let $A$ be a \cstar-algebra and $\alpha:\IZ^2\curvearrowright A$ an action with the property that $\alpha_1$ is homotopic to $\id_A$ in $\Aut(A)$. Consider the commutative diagram with split-exact rows
\begin{align}
\label{diagramObstructionTorus}
\begin{xy}
	\xymatrix{
		0 \ar[r] & K_{*-1}(A) \ar[r] \ar[d] & K_*(A\otimes \CC(\IT)) \ar[r] \ar[d]^{K_*(\phi)-\id} & K_*(A) \ar[d]^{K_{*}(\alpha_2)-\id} \ar[r] \ar@/_0.5cm/[l]& 0 \\
		0 \ar[r] & K_{*-1}(A)\ar[r] & K_*(A\otimes \CC(\IT)) \ar[r] & K_*(A) \ar[r] \ar@/^0.5cm/[l]& 0
	}
\end{xy}
\end{align}
obtained from \eqref{commutativeAlpha0}. Then the associated Snake Lemma homomorphism $h_*$ coincides with $d_2^{0,*}$.
\end{prop}
\begin{proof}
The isomorphism $\psi:\CM_{\alpha_1}(A)\stackrel{\cong}{\longrightarrow} A\otimes \CC(\IT)$ and \eqref{commutativeAlpha0} induce the following commutative diagram
\[
\xymatrix@C-0.05cm{
0\ar[rr]&& SA \ar[rr] \ar[dd]|(0.5)\hole^/-0.5cm/{\phi^{\prime}} && A\otimes \CC(\IT) \ar[rr] \ar[dd]|(0.5)\hole^/-0.5cm/{\phi} & & A \ar[dd]|(0.5)\hole^/-0.5cm/{\alpha_2} \ar[r] & 0 \\
0 \ar[r] & SA \ar[rr] \ar[dd]^/-0.5cm/{S\alpha_2} \ar[ru]^{\psi^{\prime}} && \CM_{\alpha_1}(A) \ar[rr] \ar[dd]^/-0.5cm/{\tilde{\alpha}_2} \ar[ru]^{\psi} && A \ar[dd]^/-0.5cm/{\alpha_2} \ar[rr] \ar@{=}[ru] &&0 \\
0\ar[rr]|(0.47)\hole &&	SA \ar[rr]|(0.44)\hole && A\otimes \CC(\IT) \ar[rr]|(0.6)\hole && A \ar[r] &0 \\
0 \ar[r]  &	SA \ar[rr] \ar[ru]_{\psi^{\prime}} && \CM_{\alpha_1}(A) \ar[ru]_{\psi} \ar[rr] && A \ar@{=}[ru] \ar[rr] &&0
}
\]
Observe that all occurring rows are split-exact. Apply $K$-theory to the whole diagram and use Bott periodicity for the left hand square involving the suspensions of $A$. The Snake Lemma homomorphism of the resulting front diagram is $d_*(\alpha)=d_2^{0,*}$, and the one of the back side diagram is $h_*$. The naturality of the Snake Lemma and the fact that $\psi^\prime$ acts trivially on $K$-theory yield that these two Snake Lemma homomorphisms coincide.
\end{proof}

Let now $A$ be a Kirchberg algebra and $\alpha:\IZ^2\curvearrowright A$ an action. We follow \cite{IzumiMatui2010} and say that $\alpha$ is \emph{locally $KK$-trivial} if $KK(\alpha_1)=KK(\alpha_2)=1_A$. Moreover, we call two actions $\alpha,\beta:\IZ^2\curvearrowright A$ \emph{$KK$-trivially cocycle conjugate} if there exists an $\alpha$-cocyle $u$, that is, a map $u:\IZ^{2}\to \CU(\CM(A))$ satisfying $u_g\alpha_g(u_h)=u_{gh}$ for all $g,h\in \IZ^2$, and an automorphism $\mu\in \Aut(A)$ with $KK(\mu)=1_A$ such that $\ad(u_{e_i})\circ\alpha_i=\mu\circ\beta_i\circ\mu^{-1}$ for $i=1,2$.

Given a locally $KK$-trivial action $\alpha$ on a Kirchberg algebra $A$, Izumi and Matui associate an element $\Phi(\alpha)\in KK(A,SA)$ as follows. Since $KK(\alpha_1)=1_A\in KK(A,A)$, \cite[Theorem 4.1.1]{Phillips2000} yields a homotopy $\beta_t\in \Aut(A\otimes \CK)$ between $\beta_0=\alpha_1\otimes\id_{\CK}$ and $\beta_1=\id_{A\otimes\CK}$. As above, we define the automorphism $\phi\in \Aut(A\otimes \CK\otimes \CC(\IT))$ by
\[
\phi(f)(\exp(2\pi it))=\left(\beta_{t}\circ(\alpha_{2}\otimes\id_{\CK})\circ\beta^{-1}_{t}\right)(f(\exp(2\pi it)))
\]
for $f\in A\otimes \CK\otimes \CC(\IT)$ and $t\in[0,1]$.
Denote by $j:A\otimes\CK \into A\otimes\CK\otimes \CC(\IT)$ the canonical embedding. Using the stability of the $KK$-bifunctor and the fact that $KK(\alpha_2)=1_A$, we obtain that
\[
\Phi(\alpha):=KK(\phi\circ j)-KK(j)\in KK(A,SA)\subset KK(A,A\otimes \CC(\IT)).
\]
If
\[
\gamma_*:KK(A,SA)\to \Hom(K_*(A),K_{*-1}(A)),
\]
denotes the natural homomorphism, then Izumi's and Matui's result is given as follows.

\begin{theorem}[\cite{IzumiMatui2010}]
\label{IzumiMatui}
Let $A$ be a unital Kirchberg algebra. The assignment $\alpha\mapsto \Phi(\alpha)$ induces a well-defined bijection between the following two sets:
\begin{enumerate}[i)]
	\item $KK$-trivial cocycle conjugacy classes of locally $KK$-trivial outer $\IZ^2$-actions on $A$.
	\item $\left\lbrace x\in KK(A,SA) \ :\ \gamma_0(x)([1])=0\in K_1(A)\right\rbrace$.
\end{enumerate}
If $A$ is a stable Kirchberg algebra, then the statement remains true when we take $KK(A,SA)$ as a classifying invariant.
\end{theorem}

By the definition of $\gamma_*$,
\[
\gamma_*(\Phi(\alpha)) = K_*(\phi\circ j)-K_*(j)
 = (K_*(\phi)-\id)\circ K_*(j)\in \Hom(K_*(A),K_{*-1}(A))
\]
for any locally $KK$-trivial $\IZ^2$-action $\alpha$. Hence, $\gamma_*(\Phi(\alpha))$ is the Snake Lemma homomorphism of diagram \eqref{diagramObstructionTorus} applied to $A\otimes\CK$ and $\alpha\otimes \id_\CK$. Proposition \ref{ObstructionMapCircle} therefore yields that $\gamma_*(\Phi(\alpha))=d_2^{0,*}$. Combining this observation with Theorem \ref{IzumiMatui}, we draw the following consequence.

\begin{cor}
Let $A$ be a unital Kirchberg algebra satisfying the UCT. Let $\eta_*: K_*(A)\to K_{*-1}(A)$ be a homomorphisms with $\eta_0([1])=0$. Then there is a locally $KK$-trivial $\IZ^2$-action $\alpha$ on $A$ such that the associated second page differential $d_2$ satisfies $d_2^{0,*}=\eta_*$. Moreover, $K_*(A\rtimes_\alpha \IZ^2)$ fits into a six-term exact sequence
\[
\xymatrix@C+0.2cm{
K_1(A)\oplus K_0(A) \ar[r]^{\eta_1\oplus 0} & K_0(A)\oplus K_1(A) \ar[r] & K_0(A\rtimes_\alpha\IZ^2)\ar[d]\\
\ar[u] K_1(A\rtimes_\alpha\IZ^2) & K_1(A)\oplus K_0(A) \ar[l] & K_0(A)\oplus K_1(A) \ar[l]^{\eta_0\oplus 0}
	}
\]
If $A$ is a stable Kirchberg algebra satisfying the UCT, then the statement remains true if the condition on the class of the unit is removed.
\end{cor}
\begin{proof}
Since $A$ satisfies the UCT, we find some element $x\in KK(A,SA)$ satisfying $\gamma_*(x)=\eta_*$. Observe that if $A$ is unital, then the condition $\gamma_*(x)([1])=0$ is satisfied by assumption. Theorem \ref{IzumiMatui} yields a locally $KK$-trivial action $\alpha$ with $\Phi(\alpha)=x$, and hence
\[
\eta_* = \gamma_*(x) = \gamma_*(\Phi(\alpha)) = d_2^{0,*}.
\]

As $\alpha_1\otimes \id_\CK$ is homotopic to $\id_{A\otimes \CK}$, we have that $K_i(A\rtimes_{\alpha_1}\IZ)\cong K_0(A)\oplus K_1(A)$ for $i=0,1$. The claim now follows from the Pimsner-Voiculescu sequence for $(A\rtimes_{\alpha_1}\IZ,\check{\alpha}_2,\IZ)$ and the fact that $d_2^{0,*} = d_*(\alpha)$.
\end{proof}

%%%%%%%%%%%%%%%%%%%%%%%%%%%%%%%%%%%%%%%%%%%%%%%%%%%%%%%%%%%%%%%%%%%%%%%%%%%%%%%%%%%%%%%%%%%%%%

\section{Pointwise inner $\IZ^2$-actions with non-trivial second page differentials}
\noindent
\begin{definition}
Let $A$ be a unital \cstar-algebra and $n\in \IN$. We say that an action $\alpha:\IZ^n\curvearrowright A$ is \emph{pointwise inner} if $\alpha_i$ is an inner automorphisms for $i=1,\ldots,n$. If $n=2$, $\alpha_1=\ad(v)$, and $\alpha_{2}=\ad(w)$ for some unitaries $v,w\in \CU(A)$, then we call $u(\alpha):=v^*w^*vw\in \CU(A)$ the \emph{commutator} associated with $\alpha$.
\end{definition}

It is easy to check that $u(\alpha)\in \CZ(A)$ and that $u(\alpha)$ does not depend on the choice of the implementing unitaries.

Next, we use Theorem \ref{obstructiond0} and \ref{obstructiond1} to give a description of the second page differential associated with a pointwise inner $\IZ^2$-action. 

\begin{cor}
\label{obstructionMapPointwiseInner}
Let $A$ be a unital \cstar-algebra and $\alpha:\IZ^2\curvearrowright A$ a pointwise inner action. Let $n\geq1$, $x\in \CU_n(A)$, and $p\in \CP_n(A)$. Then the associated second page differential $d_2^{0,*}:K_*(A)\to K_{*-1}(A)$ is given by
\[
d_2^{0,0}([p])=\kappa(p,u(\alpha)^{(n)})\quad \text{and}\quad d_2^{0,1}(\alpha)([x])=\kappa(u(\alpha)^{(n)},x).
\]
\end{cor}
\begin{proof}
One computes that
\[
w^*\alpha_2(v)^*\alpha_1(w)v = w^*wv^*w^*vwv^*v = v^*w^*vw = u(\alpha).
\]
As the automorphisms $\alpha_1=\op{Ad}(v)$ and $\alpha_2=\op{Ad}(w)$ satisfy
\[
\alpha_1^{(n)}(p)=v^{(n)}p{v^{(n)}}^*\quad \text{and}\quad \alpha_2^{(n)}(p)=w^{(n)}p{w^{(n)}}^*,
\]
Theorem \ref{obstructiond0} implies that
\[
d_2^{0,0}([p])=\kappa(p,u(\alpha)^{(n)}).
\]
The proof for $d_2^{0,1}([x])$ follows similarly from Theorem \ref{obstructiond1}.
\end{proof}

Consequently, $d_2$ can only be non-trivial if the unitary $u(\alpha)$ has full spectrum. Otherwise, $u(\alpha)$ is connected to $1$ by unitaries in $\mathrm{C}^*(u(\alpha))\subseteq \CZ(A)$. In fact, we have the following result.

\begin{prop}
\label{isomorphicKtheoryCentralHomotopies}
Let $A$ be a unital \cstar-algebra and $\alpha:\IZ^n\curvearrowright A$ a pointwise inner action. Assume that for $i,j\in\set{1,\ldots,n}$, the commutator associated with the $\IZ^2$-action generated by $\alpha_i$ and $\alpha_j$ is homotopic to $1$ in $\CU(\CZ(A))$. Then $K_*(A\rtimes_{\alpha}\IZ^n)\cong K_*(A\otimes \CC(\IT^n))$.
\end{prop}
\begin{proof}
The proof goes by induction over $n\in \IN$. For a single automorphism, this is trivial since $A\rtimes_{\ad(v)}\IZ\cong A\otimes \CC(\IT)$. Assume now that the statement is true for $n-1$. Denote by $\check{\alpha}$ the $\IZ^{n-1}$-action on $A\rtimes_{\alpha_{1}}\IZ$ induced by $\check{\alpha}_2,\ldots,\check{\alpha}_n$. For $i,j\in \set{1,\ldots,n}$, let $v(i,j)$ denote the commutator associated with the $\IZ^2$-action generated by $\alpha_i$ and $\alpha_j$. As inner automorphisms fix the center pointwise, it holds that $v(i,j)\in \CZ(A)\subseteq \CZ(A\rtimes_{\alpha_1}\IZ)$. Hence, $\alpha_2,\ldots,\alpha_n$ give rise to a pointwise inner action $\alpha^\prime:\IZ^{n-1}\curvearrowright A\rtimes_{\alpha_1}\IZ$. Observe that $\alpha^\prime$ is as in the statement, so that we can apply the induction hypothesis to it. For $i=1,\ldots,n$, we therefore find homotopies $w_{t,i}\in \mathcal U(\mathcal Z(A))$, $t\in [0,1]$, with $w_{0,i}=v(i,1)$ and $w_{1,i}=1$. Since these homotopies lie in the centre of $A\rtimes_{\alpha_1}\IZ$, we can define automorphisms
\[
\phi_{t,i}:A\rtimes_{\alpha_1}\IZ\stackrel{\cong}{\longrightarrow} A\rtimes_{\alpha_1}\IZ,\quad \phi_{t,i}(a)=\alpha_i(a),\ \phi_{t,i}(u)=w_{t,i}u.
\]
For all $s,t\in [0,1]$ and $i,j=2,\ldots,n$, one computes that
\[
\phi_{s,i}\circ \phi_{t,j}(u)=w_{t,j}w_{s,i}u=w_{s,i}w_{t,j}u=\phi_{t,j}\circ \phi_{s,i}(u).
\]
Thus, the $\phi_{t,i}$ define a homotopy between the $\IZ^n$-actions $\check{\alpha}$ and $\alpha^{\prime}$. In particular, the corresponding crossed products have isomorphic $K$-theory. By the induction hypothesis, we therefore obtain that
\[
\begin{array}{lclcl}
K_*(A\rtimes_\alpha\IZ^n) & \cong & K_*(A\rtimes_{\alpha_1}\IZ\rtimes_{\check{\alpha}}\IZ^{n-1})
 & \cong & K_*(A\rtimes_{\alpha_1}\IZ \rtimes_{\alpha^\prime}\IZ^{n-1})\\
 & \cong & K_*((A\rtimes_{\alpha_1}\IZ)\otimes \CC(\IT^{n-1}))
 & \cong &  K_*(A\otimes \CC(\IT^n)).
\end{array}
\]
\end{proof}

The next result shows that there are certain restrictions on second page differentials associated with pointwise inner $\IZ^2$-actions.

\begin{prop}
\label{limitationsObstructionInner}
Let $A$ be a unital \cstar-algebra and $\alpha:\IZ^2\curvearrowright A$ a pointwise inner action. Then the associated second page differential $d_2^{0,*}$ satisfies $d_2^{0,*-1}\circ d_2^{0,*} = 0$.
\end{prop}
\begin{proof}
Using the canonical isomorphism 
\[
K_*(\CC(\IT)\oplus \CC(\IT))\cong K_*(\CC(\IT))\oplus K_*(\CC(\IT)),
\]
we compute
\[
\kappa(z\oplus z, z\oplus 1)  =  \kappa(z,z)\oplus \kappa(z,1)  =  0 \hspace{0.1cm}\in K_0(\CC(\IT))\oplus K_0(\CC(\IT)).
\]
Given a projection $p\in \CP_n(A)$, there exists a unital $*$-homomorphism $\varphi:\CC(\IT)\oplus \CC(\IT)\to M_n(A)$ satisfying $\varphi(z\oplus z)=u(\alpha)^{(n)}$ and $\varphi(1\oplus 0)=p$. By Corollary \ref{obstructionMapPointwiseInner}, we get that
\[\def\arraystretch{1.3}
\begin{array}{lcl}
d_2^{0,-1}(\kappa(p,u(\alpha)^{(n)})) & = & \kappa(u(\alpha)^{(n)},pu(\alpha)^{(n)}p+1_n-p)\\
 & = & K_0(\varphi)(\kappa(z\oplus z,z\oplus 1))\\
 & = & 0.
\end{array}
\] 
This shows that $d_2^{0,-1}\circ d_2^{0,0}=0$. Applying this to $(A\otimes \CC(\IT),\alpha \otimes \id,\IZ^2)$, we conclude that the second page differential corresponding to $(SA,S\alpha,\IZ^2)$ also has this property. Bott periodicity therefore shows that $d_2^{0,0}\circ d_2^{0,1}=0$.
\end{proof}

\subsection{A natural action on the group \cstar-algebra of the discrete Heisenberg group $H_3$}
\noindent
Recall the discrete Heisenberg group 
\[
H_3:=\left\langle r,s\ :\ rsr^{-1}s^{-1}\ \text{is central}\right\rangle
\]
and its associated (full) group \cstar-algebra 
\[
\mathrm{C}^*(H_3) := \mathrm{C}^*(u,v\ \text{unitaries}\ :\ uvu^*v^*\ \text{is central}).
\]
Consider the pointwise inner $\IZ^2$-action $\alpha$ on $\mathrm{C}^*(H_3)$ given by $\alpha_1=\op{Ad}(u)$ and $\alpha_2=\ad(v)$. The associated \cstar-dynamical system $(\mathrm{C}^*(H_3),\alpha,\IZ^2)$ is universal in the following sense. Let $B$ be a unital \cstar-algebra and $\beta:\IZ^2\curvearrowright B$ a pointwise inner action. If $\beta_1=\ad(x)$ and $\beta_2=\ad(y)$, then there is a unital and equivariant $*$-homomorphism $\varphi:(\mathrm{C}^*(H_3),\alpha,\IZ^2)\to (B,\beta,\IZ^2)$ satisfying $\varphi(u)=x$ and $\varphi(v)=y$.

The Heisenberg group also admits the following description as a semidirect product
\[
H_3=\IZ^2\rtimes_{\tilde{\sigma}}\IZ,\quad \text{with}\ \tilde{\sigma}(e_1)=e_1,\ \tilde{\sigma}(e_2)=e_1+e_2.
\]
Hence, the $*$-automorphism $\sigma\in \Aut(\CC(\IT^{2}))$ satisfying $\sigma(z_1)=z_1$ and $\sigma(z_2)=z_1z_2$ gives rise to an isomorphism
\[
\mathrm{C}^*(H_3)\stackrel{\cong}{\longrightarrow} \CC(\IT^2)\rtimes_\sigma\IZ,\quad u\mapsto u,\ v\mapsto z_2,
\]
where $u\in \CC(\IT^2)\rtimes_\sigma\IZ$ denotes the canonical unitary implementing $\sigma$. Using this identification, the action $\alpha$ is given by $\alpha_1=\ad(u)$ and $\alpha_2=\ad(z_2)$, and the commutator associated with $\alpha$ satisfies
\[
u(\alpha) = u^*z_2^*uz_2 = z_1z^*_2z_2 = z_1 \in \CC(\IT^2)\rtimes_\sigma \IZ.
\]

The Bott projection $e:=e(z_1,z_2)\in M_2(\CC(\IT^{2}))$ is unitarily equivalent to $\sigma^{(2)}(e)$, see \cite[Section 1]{AndersonPaschke1989}. So, we find a unitary $x\in \CU_2(\CC(\IT^2))$ with $\sigma^{(2)}(e)=xex^*$, which in turn gives rise to a Bott element $\kappa(e,x^*u^{(2)})\in K_1(\mathrm{C}^*(H_3))$. 

Let us recall the $K$-theory of $\mathrm{C}^*(H_3)$, which was determined in \cite[Proposition 1.4(a)]{AndersonPaschke1989}. For the reader's convenience, we shall also provide a short proof.

\begin{prop}
\label{KtheoryHeisenberg}
The $K$-theory of $\mathrm{C}^*(H_3)$ is given by
\[
\begin{array}{c}
K_0(\mathrm{C}^*(H_3))=\IZ^3\left[\kappa(z_1,z_2), \kappa(z_1,u), [1]\right],\\
K_1(\mathrm{C}^*(H_3))=\IZ^3\left[ [z_2], [u], \kappa(e,x^*u^{(2)})\right].
\end{array}
\]
\end{prop}
\begin{proof}
As $K_0(\sigma)=\id$, the Pimsner-Voiculescu sequence for $\sigma$ is of the form
\[
\xymatrix@C+0.25cm{
K_0(\CC(\IT^{2})) \ar[r]^0 & K_0(\CC(\IT^{2})) \ar[r] & K_0(\CC(\IT^2)\rtimes_\sigma\IZ) \ar[d]^{\rho_0}\\
K_1(\CC(\IT^2)\rtimes_\sigma\IZ) \ar[u]^{\rho_1} & K_1(\CC(\IT^2)) \ar[l] & K_1(\CC(\IT^2)) \ar[l]^{K_1(\sigma)-\id}
}
\]
Moreover, $\kernel(K_{1}(\sigma)-\id)=\IZ[z_{1}]$. This yields 
\[
K_0(\mathrm{C}^*(H_3))\cong K_1(\mathrm{C}^*(H_3)) \cong \IZ^3.
\]
We also conclude that the natural inclusions $K_0(\CC(\IT^2))\into K_0(\CC(\IT^2)\rtimes_\sigma\IZ)$  and $\IZ[z_2]\into K_1(\CC(\IT^2)\rtimes_\sigma\IZ)$ are split-injective. Proposition \ref{index1} implies that $\rho_1([u])=[1]$ and $\rho_1(\kappa(e,x^*u^{(2)}))=[e]$. This shows the assertion for $K_1(\mathrm{C}^*(H_3))$. Analogously, the claim for $K_0(\mathrm{C}^*(H_3))$ follows since $\rho_0(\kappa(z_1,u))=[z_1]$ by Proposition \ref{index0}.
\end{proof}

\begin{theorem}
\label{HeisenbergAction}
The second page differential $d_2$ associated with the natural action $\alpha:\IZ^2\curvearrowright \mathrm{C}^*(H_3)$ is non-trivial. The $K$-theory of the crossed product $\mathrm{C}^*(H_3)\rtimes_\alpha\IZ^2$ satisfies
\[
K_0(\mathrm{C}^*(H_3)\rtimes_\alpha\IZ^2)\cong K_1(\mathrm{C}^*(H_3)\rtimes_\alpha\IZ^2)\cong\IZ^{10}.
\]
In particular, $\mathrm{C}^*(H_3)\rtimes_\alpha\IZ^2$ and $\mathrm{C}^*(H_3)\otimes \CC(\IT^2)$ are not isomorphic in $K$-theory.
\end{theorem}
\begin{proof}
By applying Corollary \ref{obstructionMapPointwiseInner} to $d_2^{0,1}:K_1(\mathrm{C}^*(H_3))\to K_0(\mathrm{C}^*(H_3))$, we obtain
\[
d_2^{0,1}([z_2]) =  \kappa(u(\alpha),z_2)  =  \kappa(z_1,z_2)\quad \text{and}\quad d_2^{0,1}([u]))=\kappa(z_1,u).
\]
The discussion in Section \ref{Section Bott} shows that the canonical trace on $\mathrm{C}^*(H_3)$ prevents $[1]\in K_0(\mathrm{C}^*(H_3))$ from being a Bott element associated with two exactly commuting unitaries. Since $\alpha$ is pointwise inner, every element in the image of $d_2^{0,1}$ is representable in such a way. Hence,
\[
\op{im}(d_2^{0,1})= \IZ^2\left[\kappa(z_1,z_2),\ \kappa(z_1,u)\right]\subset K_0(\mathrm{C}^*(H_3)),
\]
which sits inside $K_0(\mathrm{C}^*(H_3))$ as a direct summand. 

It holds that $d_2^{0,0}([1])=0$, and moreover, by Proposition \ref{limitationsObstructionInner}, we have that $d_2^{0,0}\circ d_2^{0,1}=0$. It now follows from Proposition \ref{KtheoryHeisenberg} that $d_2^{0,0}=0$. 

If $G_0:=K_0(\mathrm{C}^*(H_3))$ and $G_1:=K_1(\mathrm{C}^*(H_3))$, then the Pimsner-Voicu\-lescu sequence for $(\mathrm{C}^*(H_3)\rtimes_{\alpha_1}\IZ,\check{\alpha}_2,\IZ)$ is of the form
\begin{align}
\label{HeisenbergPV}
\begin{xy}
\xymatrix{
G_1 \oplus G_0 \ar[r]^{d_2^{0,1}\oplus 0} & G_0 \oplus G_1 \ar[r] & K_0(\mathrm{C}^*(H_3)\rtimes_{\alpha}\IZ^2) \ar[d]\\
K_1(\mathrm{C}^*(H_3)\rtimes_\alpha\IZ^2) \ar[u] & G_1 \oplus G_0 \ar[l] & G_0 \oplus G_1 \ar[l]_0
}
\end{xy}
\end{align}
The result now follows by splitting up this six-term exact sequence into two extension, and then comparing the ranks of the occurring abelian groups.
\end{proof}

We also find pointwise inner $\IZ^2$-actions on $\mathrm{C}^*(H_3)$ whose corresponding crossed products have torsion in $K$-theory.

\begin{cor}
Let $m,n\in \IN$ and denote by $\tilde{\alpha}$ the pointwise inner $\IZ^2$-action on $\mathrm{C}^*(H_3)$ generated by $\alpha_1^m$ and $\alpha_2^n$. Then
\[\def\arraystretch{1.5}
\begin{array}{l}
K_0(\mathrm{C}^*(H_3)\rtimes_{\tilde{\alpha}}\IZ^2)\cong \IZ^{10}\oplus \IZ\big/mn\IZ\oplus\IZ\big/mn\IZ,\\
K_1(\mathrm{C}^*(H_3)\rtimes_{\tilde{\alpha}} \IZ^2)\cong \IZ^{10}.
\end{array}
\]
\end{cor}
\begin{proof}
Let $d_2$ and $\tilde{d}_2$ denote the second page differentials of the spectral sequence for $\alpha$ and $\tilde{\alpha}$, respectively. One checks that $\tilde{d}_2=mn\cdot d_2$. Basically the same proof as in Theorem \ref{HeisenbergAction} shows that the $K$-theory of $\mathrm{C}^*(H_3)\rtimes_{\tilde{\alpha}}\IZ^2$ fits into the exact sequence \eqref{HeisenbergPV} with $d_2^{0,1}$ replaced by $\tilde{d}^{0,1}_2=mn\cdot d^{0,1}_2$.
\end{proof}

%%%%%%%%%%%%%%%%%%%%%%%%%%%%%%%%%%%%%%%%%%%%%%%%%%%%%%%%%%%%%%%%%%%%%%%%%%%%%%%%%%%%%%%%%%%%%%

\subsection{Certain pointwise inner actions on amalgamated free product \cstar-algebras}
\noindent
Throughout this section, $A$ denotes a unital, separable \cstar-algebra and $\alpha:\IZ^2\curvearrowright A$ a pointwise inner action whose associated commutator $u(\alpha)$ has full spectrum. Let $u,v\in A$ be unitaries satisfying  $\alpha_1=\ad(u)$ and $\alpha_2=\ad(v)$.

We start by constructing certain pointwise inner actions on amalgamated free product \cstar-algebras with non-trivial differential $d_2^{0,0}$. Let $B$ be a unital, separable \cstar-algebra whose $K$-groups both do not vanish. Also assume that there exists a central unitary $w\in \CU(\CZ(B))$, some $n\in \IN$, and a projection $p\in \CP_n(B)$ such that
\begin{align}
\label{forcedObstruction0}
\kappa(p,w^{(n)})\neq k[w]\in K_1(B)\quad \text{for all}\ k\in\IZ.
\end{align}
Observe that $w$ must have full spectrum.

Consider the two injective $*$-homomorphisms
\[
i_1:\CC(\IT)\into A,\ i_1(z):=u(\alpha)\quad \text{and}\quad i_2:\CC(\IT)\into B,\ i_2(z):=w,
\]
and form the amalgamated free product $C:=A\ast_{\CC(\IT)}B$ (see \cite[Section 10.11.11]{Blackadar1998} for a definition). There are natural unital $*$-homomorphisms $j_{1}:A\to C$ and $j_{2}:B\to C$, which are also injective by \cite[Theorem 3.1]{Blackadar1978}. Since $u(\alpha)=w$ is central in $C$, the action on $A$ extends to a pointwise inner $\IZ^{2}$-action on $C$, which we also denote by $\alpha$. The associated second page differential $d_2$ satisfies
\[
d_2^{0,0}([p])=\kappa(p,u(\alpha)^{(n)})=\kappa(p,w^{(n)})\in K_1(C).
\]

\begin{lemma}
\label{Examplesd0}
We have that $d_2^{0,0}([p])=\kappa(p,w^{(n)})\neq 0\in K_1(C)$.
\end{lemma}
\begin{proof}
By \cite[Theorem 6.1]{Thomsen2003}, there exists a six-term exact sequence
\begin{align}
\label{sixTermAmalgamated}
\begin{xy}	
\xymatrix@C+0.2cm{
K_0(\CC(\IT)) \ar[rr]^/-0.8em/{K_0(i_1)\oplus K_0(i_2)}& & K_0(A)\oplus K_0(B)\ar[rr]^/1em/{K_0(j_1)-K_0(j_2)} & & K_0(C)\ar[d] \\					
 K_1(C) \ar[u] & & K_1(A)\oplus K_1(B)\ar[ll]_/1em/{K_1(j_1)-K_1(j_2)} & & K_1(\CC(\IT)) \ar[ll]_/-0.8em/{K_1(i_1)\oplus K_1(i_2)}
}
\end{xy}
\end{align}
Since 
\[
(K_1(j_1)-K_1(j_2))(0\oplus -\kappa(p,w^{(n)}))=\kappa(p,w^{(n)}),
\]
it suffices to check that $0\oplus-\kappa(p,w^{(n)})\notin \im(K_1(i_1)\oplus K_1(i_2))$. We have that $[u(\alpha)]=0\in K_1(A)$, and therefore
\[
\im(K_1(i_1)\oplus K_1(i_2))=0\oplus \IZ[w].
\]
By assumption, $\kappa(p,w^{(n)})\neq k[w]$ for all $k\in\IZ$, and the proof is complete.
\end{proof}

As the conditions on $A$ and $B$ are very mild, Lemma \ref{Examplesd0} applies in many situations. We would like to discuss one example, which is of particular interest. Take $A:=C^*(H_{3})$ and equip it with the natural action $\alpha$ defined in the last subsection. Let $B:=\CC(\IT)\oplus \CC(\IT)$ and set $w:=z\oplus z$ and $p:=1\oplus 0$. Observe that these elements satisfy \eqref{forcedObstruction0}. Define the amalgamated free product $C_1:=A\ast_{\CC(\IT)}B$ and consider the pointwise inner action on $C_1$ induced by $\alpha:\IZ^2\curvearrowright A$, which we also denote by $\alpha$. By Lemma \ref{Examplesd0}, $\kappa(p,w)\neq 0\in K_1(C_1)$.

The \cstar-dynamical system $(C_1,\alpha,\IZ^2)$ is universal for $K_1$-obstructions for second page differentials associated with pointwise inner $\IZ^2$-actions in the following sense. For every unital \cstar-algebra $D$, any pointwise inner action $\gamma:\IZ^2\curvearrowright D$ with $\gamma_1=\ad(\bar{u})$ and $\gamma_2=\ad(\bar{v})$, and every projection $\bar{p}\in D$, there is a unital and equivariant $*$-homomorphism 
\[
\varphi:(C_1,\alpha,\IZ^2)\to (D,\gamma,\IZ^2),\quad u\mapsto \bar{u},\ v\mapsto \bar{v},\ p\mapsto \bar{p}.
\]
By naturality of the Baum-Connes spectral sequence,
\[
K_1(\varphi)(d_2^{0,0}([p]))=\tilde{d}_2^{0,0}([\bar{p}]),
\]
where $\tilde{d}_2$ denotes the second page differentials associated with $\gamma$. Therefore, we can think of $d_2^{0,0}([p])=\kappa(p,u(\alpha))\in K_1(C_1)$ as the universal $K_1$-obstruction for second page differentials associated with pointwise inner $\IZ^{2}$-actions.

The universal property of $(C_1,\alpha,\IZ^2)$ also yields that $\kappa(p,u(\alpha))\in K_1(C_1)$ has infinite order and induces a split-injection $\IZ\left[\kappa(p,u(\alpha))\right]\into K_1(C_1)$. To see this, consider the \cstar-dynamical system $(\mathrm{C}^*(H_3)\otimes \CO^\infty,\alpha\otimes \id, \IZ^2)$, where $\CO^\infty$ is the (unique) UCT Kirchberg algebra with $K_0(\CO^\infty)=0$ and $K_1(\CO^\infty)\cong \IZ$. The proof of Theorem \ref{HeisenbergAction} and the fact that $\mathrm{C}^*(H_3)\otimes \CO^\infty$ is properly infinite show that there exists a projection $q\in \mathrm{C}^*(H_3)\otimes \CO^{\infty}$ such that the cyclic subgroup
generated by
\[
0\neq d_2^{0,0}([q])\in K_1(\mathrm{C}^*(H_3)\otimes \CO^{\infty})
\]
sits inside $K_1(\mathrm{C}^*(H_3)\otimes \CO^{\infty})\cong \IZ^3$ as a direct summand. Hence, the universal property of $(C_1,\alpha,\IZ^2)$ applied to $(\mathrm{C}^*(H_3)\otimes \CO^\infty,\alpha\otimes \id, \IZ^2)$ and $q$ yields the desired result.

\begin{prop}
\label{KtheoryC1}
The canonical embedding $j_1:A\into C_1$ is split-injective in $K$-theory and induces the following decompositions:
\[
K_0(C_1)\cong K_0(A)\oplus \IZ[p]\quad \text{and}\quad K_1(C_1)\cong K_1(A)\oplus \IZ[\kappa(p,u(\alpha))].
\]
In particular, $K_0(C_1)\cong K_1(C_1)\cong \IZ^4$.
\end{prop}
\begin{proof}
A short computation shows that 
\[
K_*(i_1)\oplus K_*(i_2):K_*(\CC(\IT))\to K_*(A)\oplus K_*(B)
\]
is split-injective. Hence the six-term exact sequence \eqref{sixTermAmalgamated} associated with the amalgamated free product $C_1=A\ast_{C(\IT)}B$ reduces to a split-extension
\begin{align}
\label{sequenceK1universal}
\begin{xy}
\xymatrix{
0 \ar[r] & K_*(\CC(\IT)) \ar[rr]^/-0.3cm/{K_*(i_1)\oplus K_*(i_2)} & & K_*(A) \oplus K_*(B) \ar@/_0.7cm/[ll] \ar[rr]^/0.35cm/{K_*(j_1)- K_*(j_2)} & & K_*(C_1) \ar[r] & 0.
	}
\end{xy}
\end{align}
Consequently, $K_*(C_1)$ is torsion-free. By recalling that the $K$-theory of the Heisenberg group \cstar-algebra $A$ satisfies $K_0(A)\cong K_1(A)\cong \IZ^3$, we conclude that $K_0(C_1)\cong K_1(C_1)\cong \IZ^4$.

The universal property of the amalgamated free product yields a homomorphism $\varphi:C_1\to A$ satisfying $\varphi\circ j_1=\id_A$, $(\varphi\circ j_2)(p)=1$, and $(\varphi\circ j_2)(w)=u(\alpha)$. Obviously, $\varphi$ is surjective with splitting $j_1:A\into C_1$. This shows that $K_*(j_1)$ is split-injective. Moreover, we get that $[p]\in K_0(C_1)$ has infinite order and induces a split-injection $\IZ[p]\into K_0(C_1)$. Since we already know that the analogous statement for $\kappa(p,u(\alpha))\in K_1(C_1)$ is true as well, it remains to show that $[p]$ and $\kappa(p,u(\alpha))$ both do not lie in $K_*(j_1)(K_*(A))\subseteq K_*(C_1)$.

Suppose that there is some $g\in K_0(A)$ with $K_0(j_1)(g)=[p]$. As
\[
(K_0(j_1) - K_0(j_2))(0\oplus-[1\oplus 0])=[p],
\]
exactness of \eqref{sequenceK1universal} yields the existence of some $k\in \IZ$ satisfying
\[ 
k([1]\oplus[1\oplus 1])+ g\oplus0=0\oplus -[1\oplus0].
\]
This is a contradiction, and thus $[p]\notin K_0(j_1)(K_0(A))$. A similar argument yields $\kappa(p,u(\alpha))\notin K_1(j_1)(K_1(A))$ if one uses that
\[
(K_1(j_1)-K_1(j_2))(0\oplus-[z\oplus0])=\kappa(p,u(\alpha))\in K_1(C_1).
\]
\end{proof}

\begin{theorem}
\label{K1universal}
The second page differential $d_2$ associated with $\alpha:\IZ^2\curvearrowright A$ is non-trivial. The $K$-theory of the crossed product $C_1\rtimes_\alpha\IZ^2$ satisfies
\[
K_0(C_1\rtimes_{\alpha}\IZ^2)\cong K_0(C_1\rtimes_\alpha\IZ^2)\cong \IZ^{13}. 
\]
In particular, $K_*(C_1\rtimes_{\alpha}\IZ^2)\ncong K_*(C_1\otimes \CC(\IT^2))$.
\end{theorem}
\begin{proof}
We have that
\[
d_2^{0,0}([p])=\kappa(p,u(\alpha))\quad \text{and}\quad d_2^{0,1}(\kappa(p,u(\alpha)))=0.
\]
Note that the second equality is a consequence of Proposition \ref{limitationsObstructionInner}. As $K_*(j_1):K_*(A)\to K_*(C_1)$ is split-injective by Proposition \ref{KtheoryC1}, naturality of the Baum-Connes spectral sequence yields that on $K_*(A)\subset K_*(C_1)$, $d_2^{0,*}:K_*(C_1)\to K_{*-1}(C_1)$ coincides with the second page differential associated with $(A,\alpha,\IZ^2)$. The proof of Theorem \ref{HeisenbergAction} therefore yields that 
\[
\coker(d_2^{0,0})\cong \kernel(d_2^{0,0})\cong \IZ^3\quad \text{and}\quad \coker(d_2^{0,1})\cong \kernel(d_2^{0,1})\cong \IZ^2.
\]
As in the proof of Theorem \ref{HeisenbergAction}, the statement now follows from the Pimsner-Voiculescu sequence associated with $(C_1\rtimes_{\alpha_1}\IZ,\check{\alpha}_2,\IZ)$.
\end{proof}

Let us present another instance of a \cstar-dynamical with non-trivial second page differential arising from the above construction. Whereas $(C_1,\alpha,\IZ^2)$ is interesting for its universal property, the next \cstar-dynamical system  is minimal concerning the $K$-groups of the underlying \cstar-algebra.

\begin{theorem}
There exists a unital, separable \cstar-algebra $C$ with $K_0(C)\cong K_1(C)\cong \IZ$, which admits a pointwise inner $\IZ^2$-action $\alpha$ that is pointwise homotopic to the trivial action inside $\Inn(A)$ and whose associated second page differential $d_2$ is non-trivial. The $K$-theory of the associated crossed product is given by
\[
K_0(C\rtimes_{\alpha}\IZ^2)\cong K_1(C\rtimes_\alpha\IZ^2)\cong \IZ^3.
\]
In particular, $K_*(C\rtimes_\alpha\IZ^2)\ncong K_*(C\otimes \CC(\IT^2))$
\end{theorem}
\begin{proof}
Define $A:=\mathrm{C}^*(H_3)\otimes \CO_2$ and note that this \cstar-algebra has trivial $K$-theory, see \cite[Theorem 2.3]{Cuntz1981}. Kirchberg's absorption theorem \cite{KirchbergPhillips2000} implies that $A\cong A\otimes \CZ$, where $\CZ$ denotes the Jiang-Su algebra. Hence $A$ is $K_1$-injective by \cite[Corollary 2.10]{Jiang1997}. The unitaries $u\otimes 1$ and $v\otimes 1\in A$ are therefore homotopic to $1\in\CU(A)$. By identifying $u(\alpha)\otimes 1\in A$ with $z\oplus z\in \CC(\IT)\oplus \CC(\IT)$, we can form the amalgamated free product 
\[
C:=A\ast_{\CC(\IT)}(\CC(\IT)\oplus \CC(\IT)).
\]
Consider the pointwise inner $\IZ^2$-action $\alpha$ on $C$ induced by $\ad(u\otimes 1)$ and $\ad(v\otimes 1)$, which is obviously pointwise homotopic to the trivial action inside $\Inn(A)$. A similar calculation as in the proof of Proposition \ref{KtheoryC1} shows that $K_0(C)=\IZ[p]$ and $K_1(C)=\IZ\left[\kappa(p,u(\alpha))\right]$. Moreover, $d_2^{0,0}([p])=\kappa(p,u(\alpha))$ and $d_2^{0,1}=0$. The result now follows from the Pimsner-Voiculescu sequence for $(C\rtimes_{\alpha_1}\IZ,\check{\alpha}_2,\IZ)$.
\end{proof}

Next, we present an analogous construction yielding pointwise inner  $\IZ^2$-actions on amalgamated free product \cstar-algebras with non-trivial second page differential $d_2^{0,1}$. Let $B$ be a unital, separable \cstar-algebra whose $K$-groups both do not vanish. Assume further that there is a central unitary $w\in \CU(\CZ(B))$ with full spectrum and a unitary $x\in \CU_n(B)$ such that
\begin{align}
\label{forcedObstruction1}
\kappa(w^{(n)},x)\neq k[1]\in K_{0}(B)\quad \text{for all}\ k\in\IZ.
\end{align}
The two injective $*$-homomorphisms
\[
i_1:\CC(\IT)\into A,\ i_1(z):=u(\alpha)\quad\text{and}\quad i_2:\CC(\IT)\into B,\ i_2(z):=w,
\]
give rise to an amalgamated free product $C:=A\ast_{\CC(\IT)}B$. Observe that $u(\alpha)=w$ is a central unitary in $C$, and hence $\alpha$ extends to a pointwise inner action on $C$, which we also denote by $\alpha$. Then $d_2^{0,1}:K_1(C)\to K_0(C)$ satisfies
\[
d_2^{0,1}([x])=\kappa(u(\alpha)^{(n)},x)=\kappa(w^{(n)},x)\in K_0(C_0).
\]

\begin{lemma}
\label{Examplesd1}
It holds that $d_2^{0,1}([x]) = \kappa(w^{(n)},x)\neq0\in K_0(C)$.
\end{lemma}
\begin{proof}
The proof is very similar to the one of Lemma \ref{Examplesd0}.
\end{proof}

\begin{rem}
If $[1]\in K_0(A)$ has infinite order, then Lemma \ref{Examplesd1} remains true if we replace \eqref{forcedObstruction1} by the condition that $\kappa(w^{(n)},x)\neq0\in K_0(B)$.
\end{rem}

There exists a \cstar-dynamical system $(C_0,\alpha,\IZ^2)$ which is universal for $K_0$-obstructions for second page differentials associated with pointwise inner $\IZ^2$-actions. To define it, let again $A:=\mathrm{C}^*(H_{3})$ and equip it with the natural $\IZ^2$-action $\alpha$ from the last subsection. Moreover, let $B:=\CC(\IT^2)$ and set $w:=z_1$ and $x:=z_2$. As $\kappa(w,x)=\Fb\in K_0(B)$, \eqref{forcedObstruction1} is clearly satisfied. Form the amalgamated free product $C_0:=A\ast_{\CC(\IT)}\CC(\IT^2)$, which carries the induced pointwise inner $\IZ^2$-action $\alpha$.

Universality of this system expresses in the following property. Given a unital \cstar-algebra $D$, a pointwise inner action $\gamma:\IZ^2\curvearrowright D$ with $\gamma_1=\ad(\bar{u})$ and $\gamma_2=\ad(\bar{v})$ and a unitary $\bar{x}\in \CU(D)$, there is a unital and equivariant $*$-homomorphism 
\[
\varphi:(C_0,\alpha,\IZ^2)\to (D,\gamma,\IZ^2),\quad u\mapsto \bar{u},\ v\mapsto \bar{v},\ x\mapsto \bar{x}.
\]
By the naturality of the Baum-Connes spectral sequence,
\[
K_0(\varphi)(d_2^{0,1}([x]))=\tilde{d}_2^{0,1}([\bar{x}]),
\]
where $\tilde{d}_2$ denotes the second page differential associated with $\gamma$. In this way, $d_2^{0,1}([x])=\kappa(u(\alpha),x)\in K_0(C_0)$ can be considered as the universal $K_0$-obstruction for second page differentials associated with pointwise inner $\IZ^{2}$-actions. 

Note that the proof of Theorem \ref{HeisenbergAction} shows that $\kappa(u(\alpha),x)\in K_0(C_0)$ has infinite order and that $\IZ[\kappa(u(\alpha),x)]\into K_0(C_0)$ is split-injective.

\begin{prop}
\label{KtheoryC0}
The canonical embedding $j_1:A\into C_0$ is split-injective in $K$-theory and induces the following decompositions:
\[
K_0(C_0)\cong K_0(A)\oplus \IZ[\kappa(u(\alpha),x)]\quad \text{and}\quad K_1(C_0)\cong K_1(A)\oplus \IZ[x].
\]
In particular, $K_0(C_0)\cong K_1(C_0)\cong \IZ^4$.
\end{prop}
\begin{proof}
The proof is similar to the one of Proposition \ref{KtheoryC1}.
\end{proof}

\begin{theorem}
The second page differential associated with $\alpha:\IZ^2\curvearrowright C_0$ is non-trivial. The $K$-theory of the crossed product $C_0\rtimes_\alpha\IZ^2$ satisfies
\[
K_0(C_0\rtimes_{\alpha}\IZ^2)\cong K_1(C_0\rtimes_\alpha\IZ^2)\cong \IZ^{13}.
\]
In particular, $K_*(C_0\rtimes_\alpha\IZ^2)\ncong K_*(C_0\otimes \CC(\IT^2))$.
\end{theorem}
\begin{proof}
It holds that
\[
d_2^{0,1}([x])=\kappa(u(\alpha),x)\quad \text{and}\quad d_2^{0,0}(\kappa(u(\alpha),x))=0,
\]
Note that the second equality is a consequence of Proposition \ref{limitationsObstructionInner}. As $K_*(j_1):K_*(A)\to K_*(C_0)$ is split-injective by Proposition \ref{KtheoryC0}, naturality of the Baum-Connes spectral sequence yields that on $K_*(A)\subset K_*(C_0)$, $d_2^{0,*}:K_*(C_0)\to K_{*-1}(C_0)$ coincides with the second page differential associated with $(A,\alpha,\IZ^2)$. The proof of Theorem \ref{HeisenbergAction} therefore yields
\[
d_2^{0,0}(\alpha)=0\quad \text{and}\quad \ker(d_2^{0,1})\cong\coker(d_2^{0,1})\cong\IZ.
\]
Finally, we proceed as in the proof of Theorem \ref{HeisenbergAction}, and consider the Pimsner-Voiculescu sequence for $(C_0\rtimes_{\alpha_1}\IZ,\check{\alpha}_2,\IZ)$.
\end{proof}

%%%%%%%%%%%%%%%%%%%%%%%%%%%%%%%%%%%%%%%%%%%%%%%%%%%%%%%%%%%%%%%%%%%%%%%%%%%%%%%%%%%%%%%%%%%%%%

\begin{appendix}

\section{Spectral sequences associated with finite cofiltrations of \cstar-algebras} 
\label{Appendix}

\noindent
We recall the construction of the spectral sequence associated with a finite cofiltration of \cstar-algebras. This construction is closely related to Schochet's spectral sequence \cite{Schochet1981} associated with a filtration of a \cstar-algebra by closed ideals. For further information on spectral sequences, we refer the reader to \cite{Weibel1994}.

Consider the following finite cofiltration of \cstar-algebras
\begin{align}
\label{cofiltration}
\begin{xy}
	\xymatrix{
		A=F_n \ar@{->>}[r]^{\pi_{n}} & F_{n-1} \ar@{->>}[r] & \cdots \ar@{->>}[r] & F_1 \ar@{->>}[r]^/-0.1cm/{\pi_1} & F_0 \ar@{->>}[r]^/-0.3cm/{\pi_0} & F_{-1}=0.
	}
\end{xy}
\end{align}
We trivially extend it by setting $F_k:=F_n$ for $k>n$, $F_k:=0$ for $k< -1$, and $\pi_k:=\id_{F_k}$ in either case.

The ideal $I_k:=\op{ker}(\pi_{k})$ induces a short exact sequence
\begin{align}
\label{ExtensionToCofiltration}
\begin{xy}
	\xymatrix{
		0 \ar[r] & I_k \ar[r]^{\iota_{k}} & F_k \ar[r]^/-0.1cm/{\pi_k} & F_{k-1}\ar[r]  & 0,
	}
\end{xy}
\end{align}
whose associated boundary map is denoted by $\rho^{(k)}_*:K_*(F_{k-1})\to K_{*+1}(I_k)$.

The spectral sequence we construct allows, in principle, to determine the \linebreak$K$-theory of $A$ by means of $K_*(I_k)$ and $K_*(F_k)$. We use the following standard technique due to Massey \cite{Massey1952, Massey1953}.

\begin{defprop}[Exact couple]
An \emph{exact couple} is a pair of abelian groups $A,B$ together with group homomorphisms $f:A\to A$, $g:A\to B$, and $h:B\to A$ such that the diagram
\[
	\xymatrix{
	A \ar[rr]^f  & & A \ar[dl]^/-0.2cm/g\\
		& B \ar[ul]^/0.2cm/h &
	}
\]
commutes and is exact in the sense that at each place the image of the in-going arrow coincides with the kernel of the out-going one. We shall denote such an exact couple by $(A,B,f,g,h)$. The \emph{derived couple} is the induced exact couple $(\op{im}(f),\op{ker}(g\circ h)/\op{im}(g\circ h),f,g^\prime,h^\prime)$, where
\[\def\arraystretch{1.3}
\begin{array}{lcl}
	g^{\prime}(a)=[g(\bar{a})] &  & \text{for}\ a=f(\bar{a})\in \op{im}(f),\\
	h^{\prime}([b])=h(b) &  & \text{for}\ b\in \op{ker}(g\circ h).
\end{array}
\]

A \emph{morphism} of exact couples $(\varphi,\psi):(A,B,f,g,h)\to (A^\prime,B^\prime,f^\prime,g^\prime,h^\prime)$ is a pair of group homomorphisms $\varphi:A\to A^\prime$ and $\psi:B\to B^\prime$ such that
\[
\begin{array}{lcr}
\varphi\circ f=f^\prime\circ\varphi, & \varphi\circ h=h^\prime\circ \psi, & \psi\circ g=g^\prime\circ\varphi.
\end{array}
\]
A morphism of exact couples naturally induces a morphism between the derived couples.
\end{defprop}

Let us now explain how the notion of exact couples can be used to obtain the desired spectral sequence associated with \eqref{cofiltration}. For $p,q\in \IZ$, let
\[
E^{p,q}_{1}:=K_{p+q}(I_{p})\quad \text{and}\quad D_1^{p,q}:=K_{p+q}(F_{p}),
\]
and set
\[
E_1^{*,*}:=\bigoplus\limits_{p,q\in \IZ} E_1^{p,q}\quad \text{and}\quad D_1^{*,*}:=\bigoplus\limits_{p,q\in \IZ}D_1^{p,q}.
\]
The long exact sequences associated with \eqref{ExtensionToCofiltration} give rise to an exact couple
\begin{align}
\label{ExactCouple}
\begin{xy}
	\xymatrix@C+0.5cm@R+0.5cm{
		D_1^{*,*} \ar[rr]^{K_*(\pi_*)}_{(-1,1)} & & D_1^{*,*} \ar[dl]^/-0.2cm/{\rho^{(*)}_*}_{(1,0)}\\
		& E_1^{*,*} \ar[lu]^{K_*(\iota_*)}_{(0,0)}
	}
\end{xy}
\end{align}
We have attached each arrow with a pair of numbers denoting the bidegree of the respective map. Often, the bigrading is suppressed and we just write $E_1$ and $D_1$. 

To obtain the derived exact couple, define $d_1:E_1\to E_1$ by
\[
d^{p,q}_1:=\rho^{(p+1)}_{p+q}\circ K_{p+q}(\iota_{p}):E_1^{p,q}\to E_1^{p+1,q}.
\]
Then, $d_1$ obviously has bidegree $(1,0)$ and satisfies $d_1\circ d_1=0$ by exactness of \eqref{ExactCouple}. Set
\[
E_2^{p,q}:=\op{ker}(d_1^{p,q})/\op{im}(d_1^{p-1,q})\quad \text{and}\quad D_2^{p,q}:=\op{im}(K_{p+q}(\pi_{p+1})),
\]
and define
\[\def\arraystretch{1.3}
\begin{array}{lcl}
\overline{\iota}_{p,q}:E_{2}^{p,q}\to D^{p,q}_{2}, & & [y]\mapsto K_{p+q}(\iota_{p})(y),\\
\overline{\rho}_{p,q}:D_{2}^{p,q}\to E_{2}^{p+2,q-1},& & x\mapsto [\rho^{(p+2)}_{p+q}(\bar{x})],
\end{array}
\]
where $K_{p+q}(\pi_{p+1})(\bar{x})=x$. The derived couple of \eqref{ExactCouple} is now given by
\[
\xymatrix@C+0.5cm@R+0.5cm{
D_2^{*,*} \ar[rr]^{K_*(\pi_*)}_{(-1,1)} & & D_2^{*,*} \ar[dl]^{\overline{\rho}_{*,*}}_{(2,-1)}\\
& E_2^{*,*} \ar[lu]^{\overline{\iota}_{*,*}}_{(0,0)}
}
\]
We can use the endomorphism $d_2:E_2\to E_2$ given by $d_{2}^{p,q}=\overline{\rho}_{p,q}\circ \overline{\iota}_{p,q}$ to derive another exact couple. Observe that $d_2$ has bidegree $(2,-1)$.

By repeating this procedure, we obtain a family of pairs $(E_k,d_k)_{k\geq1}$, the \emph{spectral sequence} associated with the cofiltration \eqref{cofiltration}. The abelian group $E_k$ is called the \emph{$E_{k}$-term}, and $d_k$ is called the \emph{$k$-th page differential}, which has bidegree $(k,-k+1)$.

Bott periodicity gives rise to isomorphisms $(E_{k}^{p,2q},d_k^{p,2q})\cong (E_{k}^{p,0},d_k^{p,0})$, i.e.\ group isomorphisms respecting the respective differentials. However, we will not always use these identifications, since the bookkeeping of the occurring indices is easier using the original notation.

For $m\geq n+1$ and for all $p,q\in \IZ$, the differential
\[
d^{p,q}_{m}:E^{p,q}_{m}\to E^{p+m,q-m+1}_{m}
\]
vanishes since either $E^{p,q}_{m}=0$ or $E^{p+m,q-m+1}_{m}=0$. Therefore, $E_{m}=E_{n+1}$, and we say that the spectral sequence \emph{collapses}. We define the \emph{$E_{\infty}$-term} as $E_{\infty}^{p,q}:=E_{n+1}^{p,q}$. It is connected to $K_{*}(A)$ in the following way. For $q=0,1$, consider the diagram
\[
\xymatrix{
K_q(A)=K_q(F_n) \ar[r] & K_q(F_{n-1}) \ar[r] & \cdots \ar[r] & K_q(F_{-1})=0.
}
\]
Define $\CF_p K_q(A):=\op{ker}(K_q(A)\to K_q(F_p))$ for $p=-1,\ldots,n$, and observe that this gives rise to a filtration of abelian groups
\[
\xymatrix{
0=\CF_n K_q(A) \ar@^{(->}[r] & \CF_{n-1} K_q(A) \ar@^{(->}[r] & \cdots \ar@^{(->}[r] & \CF_{-1} K_q(A)=K_q(A).
}
\]
One can now show the existence of exact sequences
\[
\xymatrix{
0 \ar[r] & \CF_p K_{p+q}(A) \ar[r] & \CF_{p-1} K_{p+q}(A) \ar[r] & E_\infty^{p,q} \ar[r] & 0,
}
\]
or in other words, there are isomorphisms
\[
E_\infty^{p,q}\cong \CF_{p-1} K_{p+q}(A)\big/\CF_p K_{p+q}(A).
\]
Hence, the $E_\infty$-term determines the $K$-theory of $A$ up to group extension problems. We say that the spectral sequence $(E_k,d_k)_{k\geq 1}$ \emph{converges} to $K_*(A)$.

\begin{rem}
\label{inductive definition differentials}
The inductive definition of $(E_k,d_k)$ using exact couples admits the following description of the differential $d^{p,q}_k:E_k^{p,q}\to E_k^{p+k,q-k+1}$. Let $[x]\in E_k^{p,q}$ be represented by $x\in K_{p+q}(I_p)$, and consider its image $K_{p+q}(\iota_p)(x)\in K_{p+q}(F_p)$ under the map induced by the natural inclusion $\iota_p:I_p\to F_p$. Since we have started in $E_k$, there is a lift $y\in K_{p+q}(F_{p+k-1})$ for $K_{p+q}(\iota_p)(x)$ under $K_{p+q}(F_{p+k-1})\to K_{p+q}(F_p)$. Then
\[
d_k^{p,q}([x])=\left[\rho^{(p+k)}_{p+q}(y)\right]\in E_k^{p+k,q-k+1}.
\]
\end{rem}

The spectral sequence associated with a finite cofiltration is natural in the following sense. Assume that we have a commutative diagram of the form
\[
\xymatrix{
A \ar@{->>}[r] \ar[d]^{\varphi} & F_{n-1} \ar@{->>}[r] \ar[d]^{\varphi_{n-1}} & \cdots \ar@{->>}[r] & F_0 \ar[r] \ar[d]^{\varphi_0} & 0 \\
B \ar@{->>}[r] & G_{n-1} \ar@{->>}[r] & \cdots \ar@{->>}[r] & G_0 \ar[r] & 0.
}
\]
We write $(E_{A,k},d_{A,k})_{k\geq1}$ and $(E_{B,k},d_{B,k})_{k\geq1}$ for the spectral sequence associated with the upper and lower cofiltration, respectively. Note that the $\varphi_k$ are uniquely determined by $\varphi$. By the naturality of $K$-theory, $\varphi$ induces a morphism between the exact couples \eqref{ExactCouple} belonging to the upper and lower cofiltration, respectively. In this way, we obtain a \emph{morphism of spectral sequences}, i.e.\ a collection of homomorphisms with bidegree $(0,0)$
\[
E_k(\varphi):E_{A,k}\to E_{B,k},\ k\geq1,
\]
which are compatible with the differentials $d_{A,k}$ and $d_{B,k}$. This in turn yields a bigraded group homomorphism $E_{\infty}(\varphi):E_{A,\infty}\to E_{B,\infty}$. Such a spectral sequence homomorphism is called an \emph{isomorphism}, if there is some $k\geq 1$ such that $E_k(\varphi)$ is an isomorphism. Note that in this case, $E_l(\varphi)$ is an isomorphism for all $l\geq k$. If $\phi$ is a $*$-isomorphism, then, of course, $(E_k(\phi))_{k\geq 1}$ is an isomorphism of spectral sequences.

Such a spectral sequence homomorphism is well-behaved with respect to convergence in the following sense. By the naturality of $K$-theory, we get a commutative diagram
\[
\xymatrix{
0 \ar@^{(->}[r] & \CF_{n-1} K_q(A) \ar@^{(->}[r] \ar[d]^{\CF_{n-1} K_q(\varphi)} & \cdots \ar@^{(->}[r] & K_q(A) \ar[d]^{K_q(\varphi)}\\
0 \ar@^{(->}[r] & \CF_{n-1} K_q(B) \ar@^{(->}[r] & \cdots \ar@^{(->}[r] & K_q(B)
	}
\]
One can show that $E_{\infty}(\varphi)$ is induced by $\CF_* K_*(\varphi)$ in the sense that for every $p=-1,\ldots,n$ and $q=0,1$, there is a commutative diagram with exact rows
\[
\xymatrix{
0 \ar[r] & \CF_p K_q(A) \ar[r] \ar[d]^{\CF_p K_q(\varphi)} & \CF_{p-1} K_q(A) \ar[r] \ar[d]^{\CF_{p-1} K_q(\varphi)} & E_\infty^{p,q} \ar[r] \ar[d]^{E_\infty^{p,q}(\varphi)} & 0\\
0 \ar[r] & \CF_p K_q(B) \ar[r] & \CF_{p-1} K_q(B) \ar[r] & E_\infty^{p,q} \ar[r] & 0
}
\]
This yields a method to determining $K_*(\varphi)$ by means of the spectral sequences whenever $\varphi:A\to B$ is a cofiltration-respecting $*$-homomorphism.
\end{appendix}

%%%%%%%%%%%%%%%%%%%%%%%%%%%%%%%%%%%%%%%%%%%%%%%%%%%%%%%%%%%%%%%%%%%%%%%%%%%%%%%%%%%%%%%%%%%%%%

\bibliographystyle{plain}
\bibliography{D:/Mathe/Projekte/sbarlak}

\begin{thebibliography}{10}

\bibitem{AndersonPaschke1989}
J.~Anderson and W.~Paschke.
\newblock {The rotation algebra}.
\newblock {\em Houston J. Math.}, 15(1):1--26, 1989.

\bibitem{Barlak2014}
S.~Barlak.
\newblock {\em {On the {$K$}-theory of Crossed Product
  {$\mathrm{C}^*$}-algebras by Actions of {$\mathbb Z^n$}}}.
\newblock PhD thesis, University of M\"{u}nster, 2014.

\bibitem{BaumConnes2000}
P.~Baum and A.~Connes.
\newblock {Geometric {$K$}-theory for Lie groups and foliations}.
\newblock {\em Enseign. Math. (2)}, 46(1-2):3--42, 2000.

\bibitem{BaumConnesHigson1994}
P.~Baum, A.~Connes, and N.~Higson.
\newblock {Classifying space for proper actions and {$K$}-theory of group
  {$\mathrm{C}^*$}-algebras}.
\newblock In {\em {$\mathrm{C}^*$}-algebras: 1943--1993 ({S}an {A}ntonio, {TX},
  1993)}, volume 167 of {\em Contemp. Math.}, pages 240--291. Amer. Math. Soc.,
  Providence, RI, 1994.

\bibitem{BellissardKellendonkLegrand2001}
J.~Bellissard, J.~Kellendonk, and A.~Legrand.
\newblock {Gap-labelling for three-dimensional aperiodic solids}.
\newblock {\em {C. R. Acad. Sci. Paris, Sér. I Math.}}, 332(6):521--525, 2001.

\bibitem{Blackadar1978}
B.~Blackadar.
\newblock {Weak expectations and nuclear {$\mathrm{C}^*$}-algebras}.
\newblock {\em {Indiana Univ. Math. J.}}, 27:1021--1026, 1978.

\bibitem{Blackadar1998}
B.~Blackadar.
\newblock {\em {{$K$}-Theory for Operator Algebras}}, volume~5 of {\em Math.
  Sci. Res. Inst. Publ.}
\newblock Springer-Verlag, New York, second edition, 1998.

\bibitem{Brown1982}
K.~S. Brown.
\newblock {\em {Cohomology of Groups}}.
\newblock Graduate Texts in Mathematics. Springer, 1982.

\bibitem{Brown1977}
L.~G. Brown.
\newblock {Stable isomorphism of hereditary subalgebras of
  {$\mathrm{C}^*$}-algebras}.
\newblock {\em Pacific J. Math.}, 71(2):335--348, 1977.

\bibitem{BrownGreenRieffel1977}
L.~G. Brown, P.~Green, and M.~A. Rieffel.
\newblock {Stable isomorphism and strong Morita equivalence of
  {$\mathrm{C}^*$}-algebras}.
\newblock {\em Pacific J. Math.}, 71(2):349--363, 1977.

\bibitem{Connes1981}
A.~Connes.
\newblock {An analogue of the Thom isomorphism for crossed products by an
  action of {$\mathbb R$}}.
\newblock {\em Adv. Math.}, 39(1):31--55, 1981.

\bibitem{Cuntz1981}
J.~Cuntz.
\newblock {{$K$}-theory for certain {$\mathrm{C}^*$}-algebras}.
\newblock {\em Ann. of Math.}, 113(1):181--197, 1981.

\bibitem{Cuntz1984}
J.~Cuntz.
\newblock {{$K$}-theory and {$\mathrm{C}^*$}-algebras}.
\newblock {\em Springer Lecture Notes in Mathematics}, 1046:55--79, 1984.

\bibitem{Dadarlat1995}
M.~Dadarlat.
\newblock {Approximately unitarily equivalent morphisms and inductive limit
  {$\mathrm{C}^*$}-algebras}.
\newblock {\em J. K-theory}, 9(2):117--137, 1995.

\bibitem{ElliottRordam1995}
G.~Elliott and M.~R{{\o}}rdam.
\newblock {Classification of certain simple infinite {$\mathrm{C}^*$}-algebras,
  II}.
\newblock {\em Comment. Math. Helv.}, 70:615--638, 1995.

\bibitem{Exel1993}
R.~Exel.
\newblock {The soft torus and applications to almost commuting matrices}.
\newblock {\em Pacific J. Math.}, 160(2):207--217, 1993.

\bibitem{HigsonKasparov2001}
N.~Higson and G.~G. Kasparov.
\newblock {{$E$}-theory and {$KK$}-theory for groups which act properly and
  isometrically on Hilbert space}.
\newblock {\em Invent. math.}, 144:23--74, 2001.

\bibitem{IzumiMatui2010}
M.~Izumi and H.~Matui.
\newblock {{$\mathbb Z^{2}$}-actions on Kirchberg algebras}.
\newblock {\em Adv. Math.}, 224(2):355--400, 2010.

\bibitem{Jiang1997}
X.~Jiang.
\newblock {Non-stable {$K$}-theory for {$\mathcal Z$}-stable
  {$\mathrm{C}^*$}-algebras}.
\newblock 1997.
\newblock Preprint,
  \href{http://arxiv.org/abs/math/9707228}{arXiv:math.OA/9707228}.

\bibitem{Kasparov1988}
G.~G. Kasparov.
\newblock {Equivariant {$KK$}-theory and the Novikov conjecture}.
\newblock {\em Invent. Math.}, 91:147--201, 1988.

\bibitem{KirchbergPhillips2000}
E.~Kirchberg and N.~C. Phillips.
\newblock {Embedding of exact $\mathrm{C}^*$-algebras in the Cuntz algebra
  {$\mathcal O_2$}}.
\newblock {\em J. Reine Angew. Math.}, 525:17--53, 2000.

\bibitem{Lafforgue2012}
V.~Lafforgue.
\newblock {La conjecture de Baum-Connes \`a coefficients pour les groupes
  hyperboliques}.
\newblock {\em J. Noncommut. Geom.}, 6:1--197, 2012.

\bibitem{Lang1997}
S.~Lang.
\newblock {\em Algebra}.
\newblock Grad. Texts in Math. Springer New York, third edition, 1997.

\bibitem{Loring1988}
T.~A. Loring.
\newblock {{$K$}-theory and asymptotically commuting matrices}.
\newblock {\em Canad. J. Math.}, 40(1):197--216, 1988.

\bibitem{Malcev1948}
A.~I. Mal'cev.
\newblock {On the embedding of group algebras in division algebras}.
\newblock {\em Dokl. Akad. Nauk SSSR (N.S.)}, 60:1499--1501, 1948.

\bibitem{Massey1952}
W.~S. Massey.
\newblock {Exact couples in algebraic topology (Parts I \& II)}.
\newblock {\em Ann. of Math.}, 56:363--396, 1952.

\bibitem{Massey1953}
W.~S. Massey.
\newblock {Exact couples in algebraic topology (Parts III, IV \& V)}.
\newblock {\em Ann. of Math.}, 57:248--286, 1953.

\bibitem{MeyerNest2006}
R.~Meyer and R.~Nest.
\newblock {The Baum-Connes conjecture via localisation of categories}.
\newblock {\em Topology}, 45(2):209--259, 2006.

\bibitem{Neumann1949}
B.~H. Neumann.
\newblock {On ordered division rings}.
\newblock {\em Trans. Amer. Math. Soc.}, 66:202--252, 1949.

\bibitem{OlesenPedersen1986}
D.~Olesen and G.~K. Pedersen.
\newblock {Partially inner {$\mathrm{C}^*$}-dynamical systems}.
\newblock {\em J. Funct. Anal.}, 66(2):262--281, 1986.

\bibitem{Paschke1983}
W.~L. Paschke.
\newblock {On the mapping torus of an automorphism}.
\newblock {\em Proc. Amer. Math. Soc.}, 88(3):481--485, 1983.

\bibitem{Phillips2000}
N.~C. Phillips.
\newblock {A classification theorem for nuclear purely infinite
  {$\mathrm{C}^*$}-algebras}.
\newblock {\em Doc. Math.}, 5:49--114, 2000.

\bibitem{PimsnerVoiculescu1980}
M.~Pimsner and D.~Voiculescu.
\newblock {Exact sequences for {$K$}-groups and {$Ext$}-groups of certain
  cross-product {$\mathrm{C}^*$}-algebras}.
\newblock {\em J. Operator Theory}, 4(1):93--118, 1980.

\bibitem{RaeburnWilliams1985}
I.~Raeburn and D.~P. Williams.
\newblock Pullbacks of {$\mathrm{C}^*$}-algebras and crossed products by
  diagonal actions.
\newblock {\em Trans. Amer. Math. Soc.}, 287(2):755--777, 1985.

\bibitem{Rieffel1990}
M.~A. Rieffel.
\newblock {Proper actions of groups on {$\mathrm{C}^*$}-algebras}.
\newblock In {\em {Mappings of Operator Algebras}}, volume~84 of {\em Progr.
  Math.}, pages 141--182. Birkh\"auser Boston, 1990.

\bibitem{RosenbergSchochet1987}
J.~Rosenberg and C.~Schochet.
\newblock {The K{\"u}nneth theorem and the universal coefficient theorem for
  Kasparov's generalized {$K$}-functor}.
\newblock {\em Duke Math. J.}, 55(2):431--474, 1987.

\bibitem{SavignenBellissard2009}
J.~Savinien and J.~Bellissard.
\newblock {A spectral sequence for the {$K$}-theory of tiling spaces}.
\newblock {\em {Ergod. Th. \& Dynam. Sys.}}, 29:997--1031, 2009.

\bibitem{Schochet1981}
C.~Schochet.
\newblock {Topological methods for {$\mathrm{C}^*$}-algebras I: Spectral
  sequences}.
\newblock {\em Pacific J. Math.}, 96(1):193--211, 1981.

\bibitem{Takai1975}
H.~Takai.
\newblock {On a duality for crossed products by {$\mathrm{C}^*$}-algebras}.
\newblock {\em J. Funct. Anal.}, 19:25--39, 1975.

\bibitem{Thomsen2003}
K.~Thomsen.
\newblock {On the {$KK$}-theory and {$E$}-theory of amalgamated free products
  of {$\mathrm{C}^*$}-algebras}.
\newblock {\em J. Funct. Anal.}, 201:30--56, 2003.

\bibitem{Weibel1994}
C.~A. Weibel.
\newblock {\em {An Introduction to Homological Algebra}}, volume~38 of {\em
  Cambridge Stud. Adv. Math.}
\newblock Cambridge University Press, 1994.

\end{thebibliography}

\end{document}